\documentclass[11pt,reqno]{amsart}
\pdfoutput=1

\usepackage{float} 

\usepackage{latexsym,amsthm,amssymb,amsmath,amscd,mathtools,hyperref,xpatch}
\mathtoolsset{showonlyrefs}
\usepackage{tikz-cd}
\usetikzlibrary{matrix,arrows,decorations.pathmorphing}
\usepackage{bm}
\usepackage{fancyhdr}
\usepackage{mathtools}

\usepackage{tcolorbox}

\usepackage[normalem]{ulem}

\usepackage[hmargin=2cm,vmargin=1.5cm]{geometry}

\usepackage{graphicx}

\usepackage[shortlabels]{enumitem}
\setlist[enumerate]{nosep}

\usepackage{xcolor}
\newcommand\redsout{\bgroup\markoverwith{\textcolor{red}{\rule[0.5ex]{2pt}{0.8pt}}}\ULon}

\newtheorem{theorem}{Theorem}[section]
\newtheorem*{theorem*}{Theorem}
\newtheorem{lemma}[theorem]{Lemma}
\newtheorem*{lemma*}{Lemma}
\newtheorem{corollary}[theorem]{Corollary}
\newtheorem{proposition}[theorem]{Proposition}

\newtheorem{remark}[theorem]{Remark}
\newtheorem{definition}[theorem]{Definition}
\newtheorem*{definition*}{Definition}
\newtheorem*{definitions*}{Definitions}

\newtheorem*{question*}{Question}

\newtheorem*{questions*}{Questions}

\newtheorem{example}[theorem]{Example}

\newcommand{\cB}{\mathcal B}
\newcommand{\cC}{\mathcal C}

\newcommand{\cF}{\mathcal F}
\newcommand{\cG}{\mathcal G}
\newcommand{\cH}{\mathcal H}

\newcommand{\cK}{\mathcal K}

\newcommand{\cO}{\mathcal O}

\newcommand{\cR}{\mathcal R}
\newcommand{\cS}{\mathcal S}
\newcommand{\cT}{\mathcal T}
\newcommand{\cU}{\mathcal U}

\newcommand{\cZ}{\mathcal Z}

\def\Cz{\mathbb{C}}

\def\Nz{\mathbb{N}}

\def\Qz{\mathbb{Q}}
\def\Rz{\mathbb{R}}
\def\Tz{\mathbb{T}}
\def\Zz{\mathbb{Z}}

\def\1z{\mathbb{1}}
\def\SEMI{\mbox{$\times\kern-2pt\vrule height5pt width.6pt \kern3pt $}}

\newcommand{\Aut}{{\rm Aut}\,}
\newcommand{\Tr}{{\rm Tr\,}}
\newcommand{\Sp}{{\rm Sp\,}}

\newcommand{\Ad}{{\rm Ad\,}}

\newcommand{\ugr}{\overline{{\rm gr}}\,}
\newcommand{\lgr}{\underline{{\rm gr}}\,}
\newcommand{\br}{{\rm br}\,}

\newcommand{\G}{\mathcal{G}}

\renewcommand{\Re}{\textup{Re}}

\newcommand{\E}{\mathcal{E}}
\newcommand{\T}{\mathcal{T}}

\newcommand{\spn}{\textup{span}}

\renewcommand{\phi}{\varphi}

\newcommand{\QI}{\textup{QI}}
\newcommand{\KMS}{\textup{KMS}}
\newcommand{\Gr}{\textup{Gr}}
\renewcommand{\inf}{\textup{inf}}

\newcommand{\pmc}{\pi_{\mathrm{mc}}} 
\newcommand{\harm}{\chi} 

\newcommand{\bullb}{\textup{\textcolor{blue}{\textbullet}}}
\newcommand{\bullg}{\textup{\textcolor{green}{\textbullet}}}
\newcommand{\bullr}{\textup{\textcolor{red}{\textbullet}}}
\newcommand{\bullo}{\textup{\textcolor{orange}{\textbullet}}}

\begin{document}

\title{C*-dynamical invariants and Toeplitz algebras of graphs}

\thispagestyle{fancy}

 \author{Chris Bruce}\thanks{C. Bruce is supported by a Banting Postdoctoral Fellowship from the Natural Sciences and Engineering Research Council of Canada (NSERC).
 This project has received funding from the European Research Council (ERC) under the European Union's Horizon 2020 research and innovation programme (grant agreement No. 817597).}
 
\address[Chris Bruce]{School of Mathematical Sciences, Queen Mary University of London, Mile End Road, E1 4NS London, United Kingdom, and
	School of Mathematics and Statistics, University of Glasgow, University Place, Glasgow G12 8QQ, United Kingdom}
\email[Bruce]{Chris.Bruce@glasgow.ac.uk}

 \author{Takuya Takeishi}\thanks{
T. Takeishi is supported by JSPS KAKENHI Grant Number JP19K14551 (2019-2023).\footnotemark[2]
 }
 \address[Takuya Takeishi]{
  Faculty of Arts and Sciences\\
  Kyoto Institute of Technology\\
  Matsugasaki, Sakyo-ku, Kyoto\\
  Japan}
 \email[Takeishi]{takeishi@kit.ac.jp}
 
\date{\today}


\maketitle

\begin{abstract}
In recent joint work of the authors with Laca, we precisely formulated the notion of partition function in the context of C*-dynamical systems. 
Here, we compute the partition functions of C*-dynamical systems arising from Toeplitz algebras of graphs, and we explicitly recover graph-theoretic information in terms of C*-dynamical invariants. In addition, we compute the type for KMS states on C*-algebras of finite (reducible) graphs and prove that the extremal KMS states at critical inverse temperatures give rise to type III$_\lambda$ factors.
Our starting point is an independent result parameterising the partition functions of a certain class of C*-dynamical systems arising from groupoid C*-algebras in terms of $\beta$-summable orbits.
\end{abstract}

\setlength{\parindent}{0cm} \setlength{\parskip}{0.3cm}

\section{Introduction}

\subsection{Context and motivation}
A C*-algebra equipped with a time evolution (that is, a point-norm continuous one-parameter automorphism group) is called a C*-dynamical system; these systems provide a general context for studying equilibrium. Many interesting classes of C*-algebras come with canonical time evolutions, and the associated C*-dynamical systems often encode a wealth of interesting information. The analysis of Kubo--Martin--Schwinger (KMS) states of a C*-dynamical system associated with algebraic, combinatorial, or number-theoretic input often reveals subtle and surprising phenomena, intimately connected with the initial data.

Each directed graph $E$ gives rise to the Toeplitz algebra $\cT C^*(E)$ \cite{FR} and the graph C*-algebra $C^*(E)$ \cite{FLR,KPRR}; the former is an extension of the latter in a canonical way. Both these C*-algebras carry canonical actions of the circle which can be lifted to time evolutions $\alpha$ and $\bar{\alpha}$, respectively, and thus give rise to C*-dynamical systems $(\T C^*(E),\alpha)$ and $(C^*(E),\bar{\alpha})$. The study of KMS states on Cuntz--Krieger algebras of finite matrices was initiated by Enomoto, Fujii, and Watatani \cite{EFW}, generalising earlier work by Olesen and Pedersen on Cuntz algebras \cite{OP}. Since then, there have been many works on KMS states for graph algebras and their Toeplitz extensions, see, for instance, \cite{aHLRS13,aHLRS15,ExLa,CarLar,Th14} and the references therein. 
The KMS-structure of the system $(\T C^*(E),\alpha)$ is typically richer than that of $(C^*(E),\bar{\alpha})$: Any KMS state of  $(C^*(E),\bar{\alpha})$ defines a KMS state of $(\T C^*(E),\alpha)$ via pull-back along the quotient map. However, every extremal KMS state of $(\T C^*(E),\alpha)$ that does not factor through $(C^*(E),\bar{\alpha})$ is necessarily of type I (cf. Section~\ref{sec:toeplitz}), so one might be inclined to think that all interesting KMS-behaviour occurs at the level of graph C*-algebras rather than their Toeplitz extensions. However, as we shall see, the presence of many extremal KMS states that are of type I means that the C*-dynamical system of $(\T C^*(E),\alpha)$ typically has many more partition functions than $(C^*(E),\bar{\alpha})$. 
We shall exploit the multitude of partition functions to extract combinatorial information about $E$ from $(\T C^*(E),\alpha)$ that is not available from $(C^*(E),\bar{\alpha})$.

KMS states on Toeplitz algebras of finite graphs have been studied by an Huef, Laca, Raeburn, and Sims \cite{aHLRS13,aHLRS15}, where an explicit parameterisation of all KMS states is given.
The case of infinite graphs has been dealt with by Carlsen and Larsen \cite{CarLar}, but things are more complicated in the infinite case, and an explicit parameterisation is only possible in some instances.

Our motivation stems from the following natural problem:
\begin{center}
Given a countable directed graph $E$, how much information about $E$ can be recovered from the C*-dynamical system $(\T C^*(E),\alpha)$? 
\end{center}

It is known that one cannot completely recover $E$ from $(\T C^*(E),\alpha)$, even when $E$ is finite, see \cite[Example~2.1]{BLRS}. However, it is not known exactly how much information is contained in $(\T C^*(E),\alpha)$.

\subsection{Approach}
A C*-dynamical system $(A,\sigma)$ may have no type I KMS states, but typically C*-dynamical systems of C*-algebras of ``Toeplitz type'' have many low-temperature KMS factor states that are of type I. The partition functions of $(A,\sigma)$ are attached to the (quasi-equivalence classes of) extremal KMS states that are of type I. Admissible triples, a notion essentially due to Connes, Consani, and Marcolli \cite{CCM}, play a fundamental role in the definition of partition function in \cite{BLT}, and finding suitable admissible triples is a key step towards computing the partition functions of a C*-dynamical system (see \cite{BLT} and also Section~\ref{sec:pf&gpoidsPrelims} below for the definitions). We refer the reader to \cite{BLT} for details on the history leading up to the conception of partition function in the setting of C*-dynamical systems. 

One fruitful avenue for recovering information from a C*-dynamical system is through the study of its partition functions. 
This approach has been particularly successful for C*-dynamical systems arising from number-theoretic considerations. For instance, the Riemann zeta function appears as the partition function of the Bost--Connes system for $\Qz$ \cite{BC}, and the Dedekind zeta function of a number field $K$ appears as the partition function of the Bost--Connes type system associated with $K$ \cite{LLN,CoMar}. The C*-dynamical systems arising from actions of congruence monoids on rings of algebraic integers \cite{Bru1,Bru2} exhibit new and interesting phenomena \cite{BLT}; indeed, these systems usually have infinitely many distinct partition functions (this phenomenon was foreshadowed already in the work of Cuntz, Deninger, and Laca \cite{CDL}). A careful analysis of these partition functions revealed that these systems contain fine class field theoretic information \cite{BLT}.

Inspired by the successes for C*-dynamical systems of number-theoretic origin, we study partition functions and admissible triples of the C*-dynamical system $(\T C^*(E),\alpha)$, and we use these to recover combinatorial information about $E$ explicitly in terms of C*-dynamical invariants. 
It is well-known that the system $(\T C^*(E),\alpha)$ has a canonical groupoid model, that is, there is an \'{e}tale groupoid $\G_E$ and a continuous $\Rz$-valued 1-cocycle $c$ on $\G_E$ such that there is a canonical isomorphism of C*-dynamical systems $(\T C^*(E),\alpha)\cong (C^*(\G_E),\sigma^c)$, where $\sigma^c$ is the time evolution on the groupoid C*-algebra $C^*(\G_E)$ determined by $c$ (see Sections~\ref{sec:pf&gpoidsPrelims} for details).
We thus begin by considering the following general problem:
\begin{center}
	What are the partition functions of the C*-dynamical system $(C^*(\G),\sigma^c)$?\\
	(Here, $\G$ is an \'{e}tale groupoid and $c$ is a continuous $\Rz$-valued 1-cocycle.) 
\end{center}

\subsection{Results and organisation of the paper}

\subsubsection{General results on C*-dynamical systems arising from groupoids}
\vspace{-0.5cm}
In the preliminary Section~\ref{sec:pf&gpoidsPrelims}, we recall the notion of partition function from \cite{BLT} (Definition~\ref{def:partitionfnc}), general results of Kumjian--Renault on KMS states and quasi-invariant measures for groupoids and their C*-algebras \cite{KumRen} (Proposition~\ref{prop:KumRen}), and the notion of $\beta$-summable orbit from \cite{Th12} (see Section~\ref{sec:betasummable}). Our first results appear in Section~\ref{sec:pfncs&gpoids} and give a complete parameterisation and explicit description of the partition functions of $(C^*(\G),\sigma^c)$ in the case that $c^{-1}(0)$ is a principal groupoid (see Proposition~\ref{prop:main} and Theorem~\ref{thm:irredrep}). We point out that this general setting includes several natural example classes in addition to the Toeplitz algebras of graphs we consider here, for example, Bost--Connes type systems \cite{LLN}, semigroup C*-algebras of right-angled Artin monoids \cite{BruLRS}, or the class of groupoid C*-algebras from dynamics studied in \cite[Section~3]{KumRen}.

\subsubsection{Type I KMS factor states and partition functions for Toeplitz algebras of graphs}
\vspace{-0.4cm}

In Section~\ref{ssec:toeplitz}, we apply our general results to systems of the form $(\T C^*(E),\alpha)$, and we obtain a parameterisation of the partition functions of $(\T C^*(E),\alpha)$ in terms of $\beta$-summable orbits in the space $E^{\leq\infty}$ of all paths (both finite and infinite) in $E$ (Theorem~\ref{thm:Toeplitz}). For $\beta>0$, the $\beta$-summable orbits are parameterised by the set $E^0_{\beta\textup{-reg}}$ of $\beta$-regular vertices of $E$ (a notion introduced by Carlsen and Larsen in \cite{CarLar}), and the partition function $Z_v(s)$ associated with a $\beta$-regular vertex $v$ is the so-called ``fixed-target partition function'' (cf. \cite{ExLa}), which is, by definition, the generating function for finite paths in $E$ that terminate at $v$.
We prove that the function $E^0\times (0,\infty)\to [1,\infty]$ given by $(v,\beta)\mapsto Z_v(\beta)$ governs certain ``cooling'' behaviour of the system $(\T C^*(E),\alpha)$ in the sense that if $\beta>0$ is such that $\sup_{v\in E^0} Z_v(\beta)<\infty$, then every KMS$_\beta$ state of $(\T C^*(E),\alpha)$ is of type I (Theorem~\ref{thm:KMSwithgrowthcond}). 
For finite graphs, we show that the condition $\sup_{v\in E^0} Z_v(\beta)<\infty$ is equivalent to $\beta>\log\rho(A_E)$, where $\rho(A_E)$ is the spectral radius of adjacency matrix $A_E$ of $E$, so we recover the results on low-temperature KMS states from \cite{aHLRS13} (cf. Proposition~\ref{prop:tree} and Remark~\ref{rmk:lowtemppar}). For infinite graphs, we expound an example from \cite{CarLar} to demonstrates that even if every vertex is $\beta$-regular, it may be that there are KMS states that are not of type I (Example~\ref{exm:CarLarExample}). 

\subsubsection{Admissible triples from ground states and graph reconstruction}
\vspace{-0.4cm}

In \cite{CarLar}, it is proven that the extremal ground stats of $(\T C^*(E),\alpha^E)$ are in canonical bijection with the set $E^0$ of vertices in $E$. We compute in Proposition~\ref{prop:groundGNS} the admissible triples of the ground states of $(\T C^*(E),\alpha^E)$ and use them to characterise the numbers $|E^nv|$ in terms of C*-dynamical invariants of $(\T C^*(E),\alpha)$, where $E^nv$ denotes the set of paths in $E$ that terminate at $v$. This allows us to recover combinatorial information about $E$ from the C*-dynamical system $(\T C^*(E),\alpha^E)$. As a consequence, we prove that if two systems $(\T C^*(E),\alpha^E)$ and $(\T C^*(F),\alpha^F)$ are isomorphic, then there is a bijection $E^0\to F^0$, $v\mapsto v'$, such that $|E^nv|=|F^nv'|$ for all $n\geq 0$ (Theorem~\ref{thm:general}).

As an application, we show that if, in addition to the time evolution, we keep track of the canonical vertex algebra $M_E$ in $\T C^*(E)$, then we can completely and explicitly reconstruct $E$ from the data $(\T C^*(E),\alpha,M_E)$ (Theorem~\ref{thm:recon}). Thus, if there is an isomorphism $(\T C^*(E),\alpha^E)\cong(\T C^*(F),\alpha^F)$ that preserves the canonical vertex algebras, then we must have $E\cong F$ (Corollary~\ref{cor:recon}). This generalises a reconstruction theorem from \cite{BLRS} to the case of infinite graphs. This generalisation has also been obtained in \cite{DEG} using completely different techniques.

\subsubsection{Geometric description of critical inverse temperatures}
\vspace{-0.4cm}

If $v$ is $\beta$-regular for some $\beta>0$, then the abscissa of convergence $\beta_v$ of the partition function $Z_v(\beta)$ lies in $\{-\infty\}\cup[0,\infty)$, with $\beta_v=-\infty$ occurring only in degenerate cases. We call such numbers the \emph{critical inverse temperatures} of the C*-dynamical system $(\T C^*(E),\alpha)$. In Section~\ref{ssec:geom}, we give a geometric description of these critical inverse temperatures. A key innovation is to attach to each vertex $v$ a rooted tree $T(E,v)$, called the directed cover of $E$ based at $v$ (Definition~\ref{def:dircover}), which is defined by resolving all cycles at $v$; it is the universal cover based at $v$ and is analogous to the universal covering space from topology.
If $v\in E^0_{\beta\text{-reg}}$ for some $\beta>0$, then $\beta_v$ is equal to the logarithm of the upper growth rate of the tree $T(E,v)$ (Proposition~\ref{prop:betac=dim}). For finite graphs, the logarithm of the upper growth rate is equal to the Hausdorff dimension of the boundary of $T(E,v)$, and this geometric description makes several properties of the critical inverse temperatures transparent.
In the case of finite graphs, the critical inverse temperatures have an interpretation in terms of the spectral radii of adjacency matrices associated with certain strongly connected components in the graph, see \cite[Corollary~5.8~\&~Lemma~7.2]{aHLRS15}, and we explain how the results in \cite{aHLRS15} are related to our geometric description (cf. Proposition~\ref{prop:tree}).

\subsubsection{Type computations for Toeplitz algebras of finite (reducible) graphs}
\vspace{-0.4cm}

For a finite graph $E$, the KMS states of $(\T C^*(E),\alpha)$ are completely parameterised in \cite{aHLRS15}. We complete the analysis of the KMS-structure of $(\T C^*(E),\alpha)$, in the case that $E$ is finite, by computing the type of all extremal KMS states of $(\T C^*(E),\alpha)$.
We prove that the extremal KMS states are either type I or type III. 
Using modifications of the results from \cite{aHLRS15}, we prove in Section~\ref{sec:finitegraphs&type} that for each $\beta>0$, the extremal KMS states that are not of type I are completely parametrized by minimal (strongly connected) components (see Theorem~\ref{thm:aHLRSmain2}; the appropriate notion of minimality is given in Definition~\ref{def:min}). In order to compute the type of the KMS state associated with the minimal strongly connected component $C$, we reduce the problem to working with the graph C*-algebra $C^*(E_C)$, where $E_C$ is the subgraph associated with the strongly connected component $C$. 

At this point, we can apply results \cite{LLNSW} to obtain that $\psi_C$ is of type III$_\lambda$, where $\lambda=e^{-s\beta_v}$, $v$ is any vertex in $C$, and $s$ is the greatest common divisor of lengths of nontrivial cycles in $E_C$ (Theorem~\ref{thm:typedetermination}).

\subsubsection{Example classes}
\vspace{-0.4cm}

Section~\ref{sec:examples} contains our analysis of several example classes. For a finite graph $E$, we use our general results to explicitly express the generating function for finite paths in $E$ in terms of C*-dynamical invariants of $(\T C^*(E),\alpha)$ (Proposition~\ref{prop:genfnc}). We then examine two specific examples of finite graphs were we concretely compute the partition functions and the type of the extremal KMS states (Examples~\ref{ex:finite1} and \ref{ex:finite2}).

In the case of infinite graphs whose vertex sets come in levels in the appropriate sense (for example, Bratteli diagrams), we show that our conclusions can be strengthened: If the systems $(\T C^*(E),\alpha^E)$ and $(\T C^*(F),\alpha^F)$ are isomorphic, then the bijection between vertex sets coming from our general result must preserve levels (Proposition~\ref{prop:BD}). 

In general, there may be a stark contrast between the amount of information about $E$ contained in the C*-algebra $\T C^*(E)$ and amount of information about $E$ contained in C*-dynamical system $(\T C^*(E),\alpha^E)$. We demonstrate this be considering opposite graphs of Bratteli diagrams (Section~\ref{sec:BDop}).

\subsubsection{$\beta$-summability for tail equivalence classes of infinite paths}
\vspace{-0.4cm}

For $\beta<0$, the $\beta$-summable orbits are either finite or arise from certain infinite paths. In Section~\ref{sec:negbeta}, we consider the problem of $\beta$-summability of infinite paths for $\beta<0$. This case is more complicated than the case of positive $\beta$. For a certain class of row-finite graphs without nonrtivial cycles, we are able to use a result by Thomsen to prove existence of an infinite-dimensional KMS state whose associated conformal measure has support containing the orbit of a special path in the graph (Proposition~\ref{prop:KMSexistence}). In general, this KMS state need not be unique, nor must it be of type I. We then give a characterisation of $\beta$-summability of infinite paths that involves a condition on the speed at which certain paths return (Theorem~\ref{thm:summable}), and we illustrate this criteria and our general construction of a KMS state by examining concrete examples (see Example~\ref{exm:negative} and Example~\ref{exm:negative2}).

\subsection*{Acknowledgements}
\vspace{-0.4cm}

We would like to thank Sergey Neshveyev and Yusuke Isono for several helpful comments about type computations. Part of this work was carried out while C. Bruce was visiting the Department of Mathematics and Statistics at the University of Victoria, and he would like to thank the department for its hospitality; he would also like to thank Marcelo Laca for several helpful comments.

\section{Preliminaries}
\label{sec:pf&gpoidsPrelims}
\subsection{Partition functions}
Let $A$ be a separable C*-algebra. A \emph{time evolution} on $A$ is a group homomorphism $\sigma\colon\Rz\to\Aut(A)$ such that for each fixed $a\in A$, the map $t\mapsto\sigma_t(a)$ is continuous; the pair $(A,\sigma)$ is called a \emph{C*-dynamical system}. The standard notion of equilibrium in this setting is given by KMS states.
For $\beta\in \Rz^*:=\Rz\setminus\{0\}$, a state $\varphi$ on $A$ is an \emph{KMS$_\beta$ state of $(A,\sigma)$}, or a \emph{$\sigma$-KMS$_\beta$ state on $A$} if 
\[
\varphi(ab)=\varphi(b\sigma_{i\beta}(a)) \quad \text{for all $\sigma$-analytic elements $a,b\in A$},
\]
see \cite[Section~5.3]{BR} and \cite[Section~8.12.2]{Ped}. A \emph{$\sigma$-KMS$_0$} state is defined to be a $\sigma$-invariant tracial state on $A$. Note that this definition of $\sigma$-KMS$_0$ state follows \cite[Section~8.12.2]{Ped}, but differs from that given in \cite[Section~5.3]{BR}. A state $\varphi$ on $A$ is a \emph{$\sigma$-ground state} if for all $\sigma$-analytic elements $a,b\in A$, the map
\[
z\mapsto \varphi(a\sigma_z(b))
\]
is bounded on the upper half-plane. We refer the reader to \cite{BR} and \cite{Ped} for details on KMS states. 

Given a C*-dynamical system $(A,\sigma)$ and $\beta\in\Rz$, we let $\KMS_\beta(A,\sigma)$ and $\Gr(A,\sigma)$ denote the closed, convex subsets of the state space of $A$ that consist of the $\sigma$-KMS$_\beta$ and $\sigma$-ground states on $A$, respectively (cf. \cite[Proposition~5.3.23]{BR}). We let $\E(\KMS_\beta(A,\sigma))$ and $\E(\Gr(A,\sigma))$ denote the sets of extreme points in $\KMS_\beta(A,\sigma)$ and $\Gr(A,\sigma)$, respectively. Given a state $\varphi$ of a C*-algebra $A$, we let $(\pi_{\phi},\cH_{\phi},\xi_{\phi})$ denote the GNS-triple associated with $\varphi$. Recall that $\phi$ is a \emph{factor state} if the von Neumann algebra $\pi_\phi(A)''$ is a factor, and the \emph{type} of the factor state $\phi$ is the type of the von Neumann algebra $\pi_\phi(A)''$.

Let $\beta\in\Rz^*$. By Lemma~\ref{lem:extremal}(i), a state $\phi\in \KMS_\beta(A,\sigma)$ is extremal if and only if it is a factor state. This leads to a useful stratification of the set $\E(\KMS_\beta(A,\sigma))$ of extreme points of $\KMS_\beta(A,\sigma)$ according to type, and we let $\E_{\rm I}(\KMS_\beta(A,\sigma))$ denote the subset of $\KMS_\beta(A,\sigma)$ consisting of extremal $\sigma$-KMS$_\beta$ states that are of type I. 

Two C*-dynamical systems $(A,\sigma^A)$ and $(B,\sigma^B)$ are \emph{isomorphic} if there exists a *-isomorphism $\theta\colon A\overset{\cong}{\to} B$ such that $\theta\circ\sigma^A_t=\sigma^B_t\circ\theta$ for all $t\in\Rz$. 
By a \emph{C*-dynamical invariant}, we mean an invariant of the system $(A,\sigma)$ that depends on $\sigma$. Typical examples include the space of extremal KMS$_\beta$ states of a particular type, the space of ground states, and the collection of partition functions associated with the extremal KMS states that are of type I.

We shall need to extend several definitions and results from \cite{BLT} to the non-unital case; we collect these reformulations here for convenience.

\begin{definition}[cf. {\cite[Section~2]{BLT}}]
	\label{def:adtriple}
Let $(A,\sigma)$ be a C*-dynamical system with $A$ separable, let $\beta_0\in(0,\infty)$, and suppose that $\varphi$ is an extremal $\sigma$-KMS$_{\beta_0}$ state on $A$ of type I. An \emph{admissible triple $(\pi,\cH,\rho)$ for $\phi$ (with respect to $\sigma$)} consists of a separable Hilbert space $\cH$, an irreducible representation $\pi\colon A\to\cB(\cH)$, and a positive trace-class operator $\rho\in\cB(\cH)$ of norm $1$ such that 
\[
\varphi(a)=\frac{\Tr(\pi(a)\rho)}{\Tr(\rho)}\quad\text{for all } a\in A.
\]
\end{definition}

\begin{definition}
	\label{def:ground&triples}
Let $(A,\sigma)$ be a C*-dynamical system with $A$ separable, and suppose $\phi$ is an extremal $\sigma$-ground state on $A$ with associated GNS-triple $(\pi_{\phi},\cH_{\phi},\xi_{\phi})$. By Lemma~\ref{lem:extremal}, $\phi$ is a pure state.  By \cite[Proposition~5.3.19(5)]{BR}, $\phi$ is $\sigma$-invariant, so that there is a unique one-parameter unitary group $\{U_t^{\phi}\}_{t\in\Rz}$ on $\cH_{\phi}$ such that $U_t^{\phi}\pi_{\phi}(a)U_{-t}^{\phi}=\pi_{\phi}(\sigma_t(a))$ for all $a\in A$ and $t\in\Rz$. By \cite[Proposition~5.3.19(5)]{BR}, there is a unique generator $H_\phi$ of the unitary group $\{U_t^{\phi}\}_{t\in\Rz}$ such that $H_\phi\geq 0$ and $0\in\Sp H_\phi$. We call $(\pi_\phi,\cH_\phi,e^{-H_\phi})$ the \emph{admissible triple} for $\phi$.
\end{definition}

\begin{proposition}[{\cite[Proposition~2.1]{BLT}}]
\label{prop:adTripleUnique}
Let $(A,\sigma)$ be a C*-dynamical system with $A$ separable, let $\beta_0\in(0,\infty)$, and suppose that $\varphi$ is an extremal $\sigma$-KMS$_{\beta_0}$ state on $A$ of type I. Then there exists an admissible triple $(\pi,\cH,\rho)$ for $\phi$; moreover, this triple is unique up to unitary equivalence in the sense that if $(\pi',\cH',\rho')$ is another admissible triple for $\phi$, then there exists a unitary $\cU\colon \cH\overset{\cong}{\to}\cH'$ such that $\cU\pi\cU^*=\pi'$ and $\cU\rho\cU^*=\rho'$.
\end{proposition}

\begin{remark}
	\cite[Proposition~2.1]{BLT} is stated for unital, separable C*-algebras. However, the proof of this proposition (and the preceding lemmas) does not require that the C*-algebra be unital, only that the KMS state $\varphi$ is a type I factor state.
\end{remark}

The following definition from \cite[Section~2]{BLT} is crucial for much of what we shall do in this paper. 

\begin{definition}[{\cite[Definition~2.4]{BLT}}]
\label{def:partitionfnc}
Let $(A,\sigma)$ be a C*-dynamical system with $A$ separable, let $\beta_0\in(0,\infty)$, and suppose that $\varphi$ is an extremal $\sigma$-KMS$_{\beta_0}$ state on $A$ of type I. Let $(\pi,\cH,\rho)$ be any admissible triple for $\phi$, and let $H:=\frac{-1}{\beta_0}\log \rho$ be the ``represented Hamiltonian'' associated with the state $\phi$.
The \emph{partition function} of $\phi$ is 
\[
Z_\phi(s):=\Tr(e^{-s H})
\]
defined for all $s\in \Cz$ such that $\Re(s)>\beta_0$.
\end{definition}

By Proposition~\ref{prop:adTripleUnique}, $Z_\phi(s)$ does not depend on the choice of admissible triple. We point out that the partition function of $\phi$ depends only on the quasi-equivalence class of $\phi$, see \cite[Remark~2.6]{BLT}.

We shall also need to deal with partition functions of KMS$_\beta$ states for $\beta<0$. Given a C*-dynamical system  $(A,\sigma)$, define a new time evolution $\sigma^-$ on $A$ by $\sigma_t^-:=\sigma_{-t}$. Then a state $\varphi$ of $A$ is a $\sigma$-KMS$_\beta$ if and only if it is a $\sigma^-$-KMS$_{-\beta}$ state. 

\begin{definition}[cf. {\cite[Remark~2.3]{BLT}}]
Let $(A,\sigma)$ be a C*-dynamical system with $A$ separable, let $\beta_0\in(-\infty,0)$, and suppose that $\varphi$ is an extremal $\sigma$-KMS$_{\beta_0}$ state on $A$ of type I. Then $\varphi$ is a $\sigma^-$-KMS$_{-\beta_0}$ state, and $-\beta_0>0$, so there exists an admissible triple $(\pi,\cH,\rho)$ for $\phi$ (with respect to $\sigma^-$). Let $H:=\frac{1}{\beta_0}\log \rho$ be the ``represented Hamiltonian'' associated with this admissible triple.
The \emph{partition function} of $\phi$ is 
\[
Z_\phi(s):=\Tr(e^{s H})
\]
defined for all $s\in \Cz$ such that $\Re(s)<\beta_0$.
\end{definition}

\subsection{KMS states and quasi-invariant measures}
\label{sec:KMSqi}
Let $\G$ denote a second-countable locally compact Hausdorff \'{e}tale groupoid. We denote by $\G^{(0)}$ the unit space of $\G$, and we let $r,s\colon \G\to\G$ denote the range and source maps on $\G$, respectively. We let $C^*(\G)$ denote the full groupoid C*-algebra of $\G$. For background on groupoids and their C*-algebras, we refer the reader to \cite{Ren} (also see \cite{SSW} for a modern treatment).

If $\Gamma$ is a group, then a $\Gamma$-valued 1-cocycle on $\G$ is simply a groupoid homomorphism from $\G$ to $\Gamma$; let $Z^1(\G,\Rz)$ denote the set of continuous $\Rz$-valued 1-cocycles on $\G$. Each $c\in Z^1(\G,\Rz)$ gives rise to a time evolution $\sigma^c$ on $C^*(\G)$ such that
\[
\sigma_t^c(f)(\gamma)=e^{itc(\gamma)}f(\gamma)\quad\text{ for all }f\in C_c(\G),\; \gamma\in\G,\text{ and }t\in\Rz.
\]

Recall that an \emph{open bisection of $\G$} is an open subset $U\subseteq \G$ such that $r$ and $s$ are injective on $U$. If $U$ is an open bisection, then we get a homeomorphism $T:=s\circ (r\vert_U)^{-1}\colon r(U)\overset{\sim}{\to}s(U)$ (here, we used that $\G$ is \'{e}tale).
If $c\in Z^1(\G,\Rz)$ and $\beta\in\Rz$ are given, then one says that a probability measure $m$ on $\G^{(0)}$ is \emph{quasi-invariant with Radon-Nikodym cocycle $e^{-\beta c}$} if for every open bisection $U$, we have $dT_*m/dm=\exp(-\beta c\circ (s\vert_U)^{-1})$, where $T_*m$ is the pushforward of $m$ along $T$.
Let $\QI_\beta(\G,c)$ denote the set of all quasi-invariant probability measures on $\G^{(0)}$ that have Radon-Nikodym cocycle $e^{-\beta c}$. We record the following properties of $\QI_\beta(\G,c)$ here for ease of reference.

\begin{proposition}
\label{prop:QIprops}
The set $\QI_\beta(\G,c)$ is a weak$^*$-closed, convex subset of the set of probability measures on $\G^{(0)}$. Moreover, $m\in\QI_\beta(\G,c)$ is an extreme point of $\QI_\beta(\G,c)$ if and only if $m$ is ergodic in the sense that for any invariant Borel set $B\subseteq \G^{(0)}$, one has either $m(B)=0$ or $m(\G^{(0)}\setminus B)=0$.
\end{proposition}
\begin{proof}
The proof that $\QI_\beta(\G,c)$ is a weak$^*$-closed is routine, and it is easy to show that $\QI_\beta(\G,c)$ is convex.
   
The claim about extremality being equivalent to ergodicity follows from \cite[Theorem~5.5]{Ch20} (observing that $m\in\QI_\beta(\G,c)$ is extremal in $\QI_\beta(\G,c)$ if and only if it is extremal in the sense defined in \cite[Section~5]{Ch20}).\end{proof}

\begin{remark}
    Note that $\QI_\beta(\G,c)$ need not be weak$^*$-compact in the case where $\G^{(0)}$ is non-compact.
\end{remark}

There is a surjective map
\begin{equation}\label{eqn:KMStomeasure}
\KMS_\beta(C^*(\G),\sigma^c)\to \QI_\beta(\G,c), \quad \varphi\mapsto m_\varphi,
\end{equation}
where $m_\varphi$ is the probability measure associated with the state $\varphi\vert_{C_0(\G^{(0)})}$ (see \cite[Theorem~1.3]{Nesh} or \cite{Ren}).
This map has a canonical section: For each $m\in \QI_\beta(\G,c)$, let $\varphi_m:=m\circ\Phi$ where $\Phi$ is the canonical conditional expectation from $C^*(\G)$ onto $C_0(\G^{(0)})$; then $\varphi_m$ is in $\KMS_\beta(C^*(\G),\sigma^c)$, and $m_{\varphi_m}=m$. Here, we view $m$ as a state on $C_0(\G^{(0)})$.

We shall need the following result which gives a sufficient condition for the map in \eqref{eqn:KMStomeasure} to be a bijection. 

\begin{proposition}[{\cite[Proposition~3.2]{KumRen}}]
	\label{prop:KumRen}
Let $\G$ be a second-countable locally compact Hausdorff \'{e}tale groupoid, and let $c\in Z^1(\G,\Rz)$ and $\beta\in\Rz$. If $c^{-1}(0)$ is a principal groupoid, then the map $\QI_\beta(\G,c)\to \KMS_\beta(C^*(\G),\sigma^c)$ given by $m\mapsto \varphi_m$, is an affine bijection.
\end{proposition}
We point out that Proposition~\ref{prop:KumRen} can also be obtained as a consequence of \cite[Theorem~1.3]{Nesh}.

\begin{remark}
	\label{rmk:extreme}		
In the setting of Proposition~\ref{prop:KumRen}, the KMS state $\varphi_m$ corresponding to $m\in\QI_\beta(\G,c)$ is an extreme point of $\KMS_\beta(C^*(\G),\sigma^c)$ if and only if $m$ is an extreme point of $\QI_\beta(\G,c)$.
\end{remark}

\subsection{$\beta$-summable orbits}
\label{sec:betasummable}
Let $\beta\in\Rz^*$. The easiest examples of quasi-invariant measures arise from certain orbits for the action of $\cG$ on $\cG^{(0)}$. For $x,y\in\G^{(0)}$, let $\G_x^y:=s^{-1}(x)\cap r^{-1}(y)$. Following Thomsen \cite[Section~2]{Th12}, we shall say that the orbit $\G x:=r(s^{-1}(x))$ of $x\in\G^{(0)}$ is \emph{consistent} if $c(\G_x^x)=\{0\}$; in this case, we can define a function $l_x:\G x\to (0,\infty)$ by $l_x(z):=e^{-c(\gamma)}$, where $\gamma\in \G_x^z$ (since $\G x$ is consistent, this definition does not depend on the choice of $\gamma$). 
A consistent orbit $\G x$ is \emph{$\beta$-summable} (see \cite[Section~2]{Th12}) if 
\begin{equation}\label{eqn:beta-summable}
\sum_{z\in \G x}l_x(z)^{\beta}<\infty.
\end{equation}
Denote by $\cO_\beta(\G,c)$ the set of $\beta$-summable orbits. Given a $\beta'$-summable orbit $O=\G x\in\cO_{\beta'}(\G,c)$, we let $\beta_O:=\inf(\{\beta\in \Rz : \sum_{z\in \G x}l_x(z)^{\beta}<\infty\})$, so that the series in \eqref{eqn:beta-summable} converges for all $\beta>\beta_O$. 

Each $\beta$-summable orbit $O=\G x$ gives rise to a quasi-invariant probability measure $m_O$ with Radon-Nikodym cocycle $e^{-\beta c}$ via
\begin{equation}
\label{eqn:measurefromorbit}
m_O:=\left(\sum_{z\in O}l_x(z)^{\beta}\right)^{-1}\sum_{z\in O}l_x(z)^{\beta}\delta_z,
\end{equation}
see \cite[Equation~2.5]{Th12} and the surrounding discussion. We let $\varphi_O:=m_O\circ \Phi$ denote the $\sigma^c$-KMS$_\beta$ state on $C^*(\G)$ corresponding to $m_O$.

\begin{remark}
    \label{rmk:consistent}
If $m$ is a quasi-invariant probability measure on $\G^{(0)}$ with Radon-Nikodym cocycle $e^{-\beta c}$ for some $\beta\in (0,\infty)$, then $\G^x_x\subseteq c^{-1}(0)$ for $m$-a.e.\ $x\in\G^{(0)}$ (cf. \cite[Remark~6.5]{Ch20}). Hence, if $c^{-1}(0)$ is principal, then $\G_x^x=\{x\}$ for $m$-a.e.\ $x\in\G^{(0)}$. In particular, if $m$ is concentrated on an orbit $O$, then $O$ must be consistent; moreover, quasi-invariance of $m$ implies that $O$ is even $\beta$-summable.
\end{remark}

\section{Partition functions and groupoid C*-algebras}
\label{sec:pfncs&gpoids}
Throughout this section, $\cG$ denotes a second-countable locally compact Hausdorff \'{e}tale groupoid and $c$ an $\Rz$-valued continuous 1-cocycle on $\cG$.

\subsection{Extremal type I KMS states arising from $\beta$-summable orbits}
\label{sec:summable}

Fix $\beta\in\Rz^*$. In the case where $c^{-1}(0)$ is principal, we shall establish a one-to-one correspondence between $\E_{\rm I}(\KMS_\beta(C^*(\G),\sigma^c))$ and the set of $\beta$-summable orbits; then we will compute the partition function associated with each $\varphi\in\E_{\rm I}(\KMS_\beta(C^*(\G),\sigma^c))$ in terms of the corresponding $\beta$-summable orbit.

\begin{proposition}
\label{prop:gns}
Suppose $m\in \QI_\beta(\G,c)$ for some $\beta\in\Rz^*$ and that $\G_x^x=\{x\}$ for $m$-a.e. $x\in\G^{(0)}$. If $\varphi_m$ is extremal in $\KMS_\beta(C^*(\G),\sigma^c)$, then $\varphi_m$ is of type I if and only if $m=m_O$ for some $\beta$-summable orbit $O$. 
\end{proposition}
\begin{proof}
Suppose $\varphi_m$ is extremal in $\KMS_\beta(C^*(\G),\sigma^c)$, so that $\varphi_m$ is a factor state. 
Using the assumption that $\G_x^x=\{x\}$ for $m$-a.e. $x\in\G^{(0)}$, the proof of parts (1) and (2) of \cite[Proposition~5.2]{LLNSW} can be adapted to show that the *-homomorphism 
$\vartheta\colon C_c(\G)\to \Cz[\cR]$ defined by 
\[ \vartheta(f)(x,y)=\sum_{\gamma\in \G_y^x}f(\gamma)\quad\text{for } f\in C_c(\G) \]
extends to an isomorphism $\pi_{\varphi_m}(C^*(\G))''\cong W^*(\cR)$, where $\pi_{\varphi_m}$ is the GNS representation of $\varphi_m$, $\cR$ is the orbit equivalence relation on $\G^{(0)}$ arising from the action of $\G$ on $\G^{(0)}$, and $W^*(\cR)$ is the von Neumann algebra of $\cR$ (with respect to $m$). We refer the reader to \cite[Section~4]{LLNSW} and \cite{FM1,FM2} for background on von Neumann algebras of equivalence relations.

By Proposition~\ref{prop:QIprops}, the equivalence relation $\cR$ on $(\G^{(0)},m)$ is ergodic. Thus, the factor $W^*(\cR)$ is of type I if and only if $m$ is concentrated on an orbit (see, for instance, \cite[Theorem~5.4]{Hahn}). If $m$ is concentrated on the orbit $O$, then $O$ must be consistent and $\beta$-summable (see Remark~\ref{rmk:consistent}).
\end{proof}

We are now ready to give sufficient conditions under which the problem of computing the extremal $\sigma^c$-KMS$_\beta$ states on $C^*(\G)$ (for $\beta>0$) that are of type I can be reduced to the problem of computing the $\beta$-summable orbits. In Theorem~\ref{thm:irredrep} below, we will use this parameterisation to compute the partition function of the KMS state attached to a $\beta$-summable orbit.

\begin{proposition}
	\label{prop:main}
	Let $\G$ be a second-countable locally compact Hausdorff \'{e}tale groupoid, and let $c\in Z^1(\G,\Rz)$ and $\beta\in (0,\infty)$. If the groupoid $c^{-1}(0)$ is principal, then the map 
	\[
	\cO_\beta(\G,c)\to \E_{\rm I}(\KMS_\beta(C^*(\G),\sigma^c)), \quad O\mapsto \varphi_O,
	\] 
	is a bijection. 
\end{proposition}
\begin{proof}
	Since $c^{-1}(0)$ is principal, Proposition~\ref{prop:KumRen} implies that the map $\QI_\beta(\G,c)\to \KMS_\beta(\G,c)$ given by $m\mapsto \varphi_m$ is a bijection. If the $\sigma^c$-KMS$_\beta$ state $\varphi_m$ is extremal, then Proposition~\ref{prop:gns} implies that $\varphi_m$ is of type I if and only if $m=m_O$ for some $O\in \cO_\beta(\G,c)$ (cf. Remark~\ref{rmk:consistent}). 
\end{proof}

\begin{remark}
	If the groupoid $c^{-1}(0)$ is principal, then by \cite[Proposition~3.2]{KumRen}, every $\sigma^c$-KMS$_\beta$ state factors through the quotient map $C^*(\G)\to C_r^*(\G)$, so that there is a canonical isomorphism $\KMS_\beta(C^*(\G),\sigma^c)\cong \KMS_\beta(C_r^*(\G),\sigma^c)$; here, we also use $\sigma^c$ to denote the time evolution on $C_r^*(\G)$ arising from $c\in Z^1(\G,\Rz)$. Thus, the statement of Proposition~\ref{prop:main} also holds if $C^*(\G)$ is replaced by $C^*_r(\G)$. In general, it is a difficult problem to determine which $\sigma^c$-KMS states on $C^*(\G)$ factor through the quotient map $C^*(\G)\to C_r^*(\G)$, see \cite{ChNesh}.
\end{remark}

\subsection{The partition function associated with a $\beta$-summable orbit}

For our computation of the partition function associated with $\phi_O$, where $O\in\cO_\beta(\cG,c)$, it will be convenient to scale the series appearing as the normalising factor in \eqref{eqn:measurefromorbit}. 

\begin{lemma}
\label{lem:series}
 Suppose $O=\G x\subseteq \G^{(0)}$ is a $\beta$-summable orbit for some $\beta\in (0,\infty)$. For each $z\in O$, choose $\gamma_z\in \G_x^z$. Then the countable set $\{c(\gamma_z) : z\in O\}$ is bounded below, and for each $r\in\Rz_{>0}$, there are only finitely many $\gamma_z$ for which $c(\gamma_z)\leq r$. If we let $\lambda_1,\lambda_2,...$ be a strictly increasing listing of the elements in $\{c(\gamma_z) : z\in O\}$, then
 \[
 Z_O(s):=e^{s\lambda_1}\sum_{z\in \G x}l_x(z)^{s}=\sum_{i=1}^\infty  |\{z\in\G x : c(\gamma_z)=\lambda_i\}|e^{-s (\lambda_i-\lambda_1)} 
 \]
  is a (generalised) Dirichlet series as defined in \cite{HR} which converges for all $s\in\Cz$ with $\Re(s)>\beta_O$.
\end{lemma}
\begin{proof}
    This is a routine calculation.
\end{proof}

We record the following easy consequence as a corollary for ease of reference.

\begin{corollary}
\label{cor:mO}
The measure $m_O$ can equivalently be written as 
\[
m_O=\frac{1}{Z_O(\beta)}\sum_{z\in O}e^{-\beta(c(\gamma_z)-\lambda_1)}\delta_z.
\]
\end{corollary}

We shall need the following result from \cite{SW} on irreducible representations coming from orbits.

\begin{lemma}[{\cite[Lemma~4.5]{SW}}]\label{lem:SW}
	Let $\G$ be a second-countable locally compact Hausdorff \'{e}tale groupoid, and let $O\subseteq \G^{(0)}$ be an orbit. There exists an irreducible representation $\pi_O\colon C^*(\G)\to\cB(\ell^2(O))$
	such that 
	\begin{equation}
	\label{eqn:irredrep}
	\pi_O(f)\delta_z=\sum_{\gamma\in\G_z}f(\gamma)\delta_{r(\gamma)}\quad \text{for all }z\in O, f\in C_c(\G).
	\end{equation}
\end{lemma}

The next theorem, which is the main result of this section, provides us with an explicit admissible triple (in the sense of Definition~\ref{def:adtriple}) attached to a $\beta$-summable orbit $O$ in the case $\beta>0$. This allows us to obtain a formula for the partition function of the KMS state attached to the $\beta$-summable orbit.

\begin{theorem}
\label{thm:irredrep}
Suppose $O=\G x\in\cO_\beta(\G,c)$ for some $\beta\in (0,\infty)$ is such that $\G_x^x=\{x\}$, and let $\pi_O\colon C^*(\G)\to\cB(\ell^2(O))$ be the irreducible representation associated with $O$ from Lemma~\ref{lem:SW}. For each $z\in O$, choose $\gamma_z\in\G_x^z$. Let $H_O$ be the (possibly unbounded) diagonal operator on $\ell^2(O)$ defined by $H_O\delta_z=(c(\gamma_z)-\lambda_1)\delta_z$, where $\lambda_1:=\min(\{c(\gamma_z) : z\in O\})$, as in Lemma~\ref{lem:series}. Then $(\pi_O,\ell^2(O),e^{-\beta H_O})$ is an admissible triple for $\phi_O$ (cf. Definition~\ref{def:adtriple}), so that the partition function, $Z_{\varphi_O}(s)$, of $\phi_O$ is given by $\Tr(e^{-s H_O})=Z_O(s)$.
\end{theorem}

\begin{proof}
Let $f\in C_c(\G)$, $z\in O$, and $t\in\Rz$. Since $O$ is consistent, we have $c(\gamma)=c(\gamma_{r(\gamma)}\gamma_z^{-1})$ for every $\gamma\in\G_z$. Hence, we have
\begin{align*}
\pi_O(\sigma^c_t(f))e^{it H_O}\delta_z=e^{it(c(\gamma_z)-\lambda_1)}\sum_{\gamma\in\G_z}e^{itc(\gamma)}f(\gamma)\delta_{r(\gamma)}&=\sum_{\gamma\in\G_z}f(\gamma)e^{it(c(\gamma_z)-\lambda_1)}e^{it(c(\gamma_{r(\gamma)})-c(\gamma_z))}\delta_{r(\gamma)}\\
&=\sum_{\gamma\in\G_z }f(\gamma)e^{it H_O}\delta_{r(\gamma)}=e^{it H_O}\pi_O(f)\delta_z,
\end{align*}
which shows that $\pi_O(\sigma^c_t(f))=\Ad(e^{it H_O})(\pi_O(f))$ for all $t\in\Rz$ and $f\in C_c(\G)$.

We can see that $H_O$ is positive and $\Sp H_O \ni 0$ directly from the definition of $H_O$. Using the orthonormal basis $\{\delta_z : z\in O\}$, it is easy to see that $\Tr(e^{-\beta H_O})=Z_O(\beta)$, so that in particular, $e^{-\beta H_O}$ is of trace class. 

If $f\in C_c(\G)$ and $z\in O$, then
\[
\langle\pi_O(f)\delta_z,\delta_z\rangle=\langle\sum_{\gamma\in\G_z}f(\gamma)\delta_{r(\gamma)}
,\delta_z\rangle=\sum_{\gamma\in\G_z^z}f(\gamma)=f(z),
\]
where the last equality uses that $\G_z^z=\{z\}$.
Now we have:
\begin{align*}
\Tr(\pi_O(f) e^{-\beta H_O})&=\sum_{z\in O}e^{-\beta (c(\gamma_z)-\lambda_1)}\langle\pi_O(f)\delta_z,\delta_z\rangle=\sum_{z\in O}e^{-\beta (c(\gamma_z)-\lambda_1)}f(z)\\
&=Z_O(\beta)\left(\frac{1}{Z_O(\beta)}\sum_{z\in O}f(z)e^{-\beta (c(\gamma_z)-\lambda_1)}\right)=Z_O(\beta)\varphi_O(f),
\end{align*}
where the last equality uses Corollary~\ref{cor:mO}. Thus, we have $\varphi_O(a)=\frac{\Tr(\pi_O(a) e^{-\beta H_O})}{\Tr(e^{-\beta H_O})}$ for all $a\in C^*(\G)$, so that $(\pi_O,\ell^2(O),e^{-\beta H_O})$ is an admissible triple for $\phi_O$. 
\end{proof}

\begin{corollary}
Suppose $O=\G x\in\cO_\beta(\G,c)$ for some $\beta\in (-\infty,0)$ is such that $\G_x^x=\{x\}$, then $O\in\cO_{-\beta}(\G,-c)$, and $-\beta\in (0,\infty)$. Let $(\pi_O,\ell^2(O),e^{\beta H_O})$ be the admissible triple from Theorem~\ref{thm:irredrep} (with respect to $\sigma^{-c}$) corresponding to the $(\sigma^{-c})$-KMS$_{-\beta}$ state $\varphi_\cO$. Then the partition function of $\varphi_\cO$ is given by $Z_{\varphi_O}(s)=Z_O(-s)=\Tr(e^{s H_O})$.
\end{corollary}

\section{C*-dynamical invariants for Toeplitz algebras of graphs ($\beta>0$)}
\label{sec:toeplitz}

\subsection{Background on C*-algebras associated with graphs}
\label{ssec:backgrounToeplitz}

In this subsection, we layout some notation and conventions that will be used throughout this paper. Our conventions for graphs will follow those used in \cite{CarLar}, rather than those from \cite{aHLRS13,aHLRS15}.
Let $E=(E^0,E^1,r,s)$ be a countable directed graph with vertex set $E^0$, edge set $E^1$, and range and source maps $r,s\colon E^1\to E^0$. 
 
Recall that $E$ is said to be finite when both $E^0$ and $E^1$ are finite.
Let $n\geq 1$. A path in $E$ of length $n$ is a sequence $\mu=\mu_0\mu_1...\mu_{n-1}$ of edges $\mu_i\in E^1$ such that $r(\mu_i)=s(\mu_{i+1})$ for $0\leq i< n-1$. We shall view vertices $v\in E^0$ as paths of length $0$, and for each $n\geq 0$, we let $E^n$ be the set of paths in $E$ that are of length $n$. Let $E^*:=\bigsqcup_{n\geq 0} E^n$ be the set of all finite paths in $E$, and let $E^\infty$ be the set of infinite paths in $E$, that is, infinite sequences of edges $x=x_0x_1...$ such that $r(x_i)=s(x_{i+1})$ for all $i\geq 0$.
For $\mu\in E^*$, let $|\mu|$ denote the length of $\mu$. We extend $s$ to $E^\infty$ in the obvious way, and also extend both $r$ and $s$ to $E^*$. Set $E^{\leq \infty}:=E^*\cup E^\infty$. Given $\mu\in E^*$ and $\nu\in E^{\leq \infty}$ such that $r(\mu)=s(\nu)$, we let $\mu\nu$ denote the path formed by concatenating $\mu$ and $\nu$. If $\mu\in E^*$ with $s(\mu)=r(\mu)$ and $n\geq 0$, then we let $\mu^n:=\underbrace{\mu\mu...\mu}_{\textup{n times}}$.

The \emph{Toeplitz algebra} $\T C^*(E)$ of $E$, as defined in \cite[\S~4]{FR}, is the universal C*-algebra generated by mutually orthogonal projections $\{p_v: v\in E^0\}$ and partial isometries $\{s_e : e\in E^1\}$ with mutually orthogonal range projections such that $s_e^*s_e=p_{r(e)}$ and  $s_es_e^*\leq p_{s(e)}$ for all $e\in E^1$. For a finite path $\mu=\mu_0...\mu_{n-1}$ of length $n\geq 1$, we put $s_\mu:=s_{\mu_0}\cdots s_{\mu_{n-1}}$. Also put $s_v:=p_v$ for $v\in E^0$. Then we have the following well-known description of $\T C^*(E)$ in terms of the partial isometries $\{s_\mu :\mu\in E^*\}$:
\[
\T C^*(E)=\overline{\spn}(\{s_\mu s_\nu^* : \mu,\nu\in E^*, r(\mu)=r(\nu)\}).
\]
Following the notation from \cite{CarLar}, we let $E_{\rm reg}^0:=\{v\in E^0 : 0<|vE^1|<\infty\}$ denote the set of \emph{regular vertices}, where $vE^1:=\{e\in E^1 : s(e)=v\}$ is the set of edges emanating from $v$. The \emph{graph C*-algebra} $C^*(E)$ of $E$ is the universal C*-algebra defined with the same generators and relations as $\T C^*(E)$ plus the additional relation
\begin{equation*}
	p_v=\sum_{e\in s^{-1}(v)}s_es_e^*\quad \text{for all }v\in E_{\rm reg}^0.
\end{equation*}
Following \cite{aHLRS15}, we shall denote the projection in $C^*(E)$ corresponding to $v\in E^0$ by $\bar{p}_v$ and the partial isometry in $C^*(E)$ corresponding to $e\in E^1$ by $\bar{s}_e$. There is a surjective *-homomorphism $\T C^*(E)\to C^*(E)$ such that $p_v\mapsto \bar{p}_v$ for all $v\in E^0$ and $s_e\mapsto \bar{s}_e$ for all $e\in E^1$.

The \emph{gauge action} on $\T C^*(E)$ is given by $\alpha^E\colon \Rz\to \Aut(\T C^*(E))$, where for each $t\in\Rz$, $\alpha_t^E(p_v)=p_v$ for all $v\in E^0$ and $\alpha_t(s_\mu)=e^{it|\mu|}s_\mu$ for all $\mu\in E^*\setminus E^0$. When it will not cause confusion, we shall omit the superscript and write $\alpha$ instead of $\alpha^E$. There is also a \emph{gauge action} on $C^*(E)$ given by $\bar{\alpha}\colon\Rz\to\Aut(C^*(E))$, where for each $t\in\Rz$, $\bar{\alpha}_t(\bar{p}_v)=\bar{p}_v$ for all $v\in E^0$ and $\bar{\alpha}_t(\bar{s}_\mu)=e^{it|\mu|}\bar{s}_\mu$ for all $\mu\in E^*\setminus E^0$. The canonical surjection $\T C^*(E)\to C^*(E)$ is clearly $\Rz$-equivariant. We write ``KMS state'' rather than ``$\alpha$-KMS state'' when referring to KMS states of $(\T C^*(E),\alpha)$ (and similarly for the system $(C^*(E),\bar{\alpha})$).

The C*-algebra $\T C^*(E)$ has several groupoid models; the most common such model is given as follows (cf. \cite{Pat}). Let
\[
\G_E:=\{(\mu x,|\mu|-|\nu|,\nu x)\in E^{\leq \infty}\times \Zz\times E^{\leq \infty}: \mu,\nu\in E^*,x\in E^{\leq \infty}, r(\mu)=s(x)=r(\nu)\}
\]
with range and source maps given by $r(x,k,y)=x$ and $s(x,k,y)=y$, respectively, and composition given by $(x,k,y)(y,l,z)=(x,k+l,z)$. For each finite path $\mu\in E^*$, let 
\[
\cZ(\mu):=\{x\in E^{\leq \infty} : x_i=\mu_i \text{ for } 0\leq i\leq |\mu|-1\},
\]
and put $\cZ'(\mu):=\cZ(\mu)\cap E^\infty$.
For each $\mu\in E^*$ and finite subset $F\subseteq r(\mu)E^1$, put $\cZ_F(\mu):=\cZ(\mu)\setminus\bigcup_{e\in F}\cZ(\mu e)$. The sets $\cZ_F(\mu)$ form a basis of compact open subsets for a locally compact Hausdorff topology on $E^{\leq \infty}$. See \cite[Theorem~2.1]{Web} (or \cite[Proposition~2.2]{CarLar}) for more on this topology. Note that the conventions for graphs in \cite{Web} are opposite to the ones used here.

The \'{e}tale topology on $\G_E$ has countable basis given by sets of the form
\[
U_{\mu,\nu}:=\{(\mu x, |\mu|-|\nu|,\nu x) : r(\mu)=s(x)=r(\nu), \mu x\in\cZ_F(\mu), \nu x\in\cZ_{F'}(\nu)\},
\]
where $\mu,\nu\in E^*$ and $F\subseteq r(\mu)E^1, F'\subseteq r(\nu)E^1$ with $F$ and $F'$ both finite. We shall usually identify $E^{\leq \infty}$ with the unit space $\G_E^{(0)}$ of $\G_E$ via $x\mapsto (x,0,x)$.

Consider the continuous $\Rz$-valued 1-cocycle $c\colon\G_E\to \Rz$ defined by $c(x,k,y):=k$, and let $\sigma^c$ be the associated time evolution on $C^*(\G_E)$. 
We shall rely on the following useful description of $(\T C^*(E),\alpha)$.
\begin{proposition}[cf. {\cite[Remark~4.4]{CarLar}}]
	\label{prop:gpoidmodel}
	There is an isomorphism of C*-dynamical systems $(\T C^*(E),\alpha)\cong (C^*(\G_E),\sigma^c)$ such that $s_e\mapsto 1_{U_{e,r(e)}}$ for every $e\in E^1$, and $p_v\mapsto 1_{U_{v,v}}$ for every $v\in E^0$. 
\end{proposition}

The graph C*-algebra, $C^*(E)$, of $E$ also has a groupoid model: There is a canonical isomorphism $C^*(E)\cong C^*(\cG_E\vert_{\partial E})$, where $\partial E\subseteq E^{\leq \infty}$ is defined by $\partial E:=E^\infty\cup\{\mu\in E^* : r(\mu)\not\in E_{\rm reg}\}$.

\begin{remark}
	To simplify things, we shall say that a measure $m$ on $E^{\leq \infty}$ is \emph{$e^\beta$-conformal} if it is quasi-invariant with Radon-Nikodym cocycle $e^{-\beta c}$ (see \cite[Section~3]{Th14}).
\end{remark}

\subsection{Extremal type I KMS states arising from vertices, and partition functions}
\label{ssec:toeplitz}

Let $E=(E^0,E^1,r,s)$ be a countable directed graph. Let $v\in E^0$. If $E^nv:=\{\mu\in E^n : r(\mu)=v\}$ is finite for all $n$, then we obtain the (generalised) Dirichlet series 
\[
Z_v^E(s):=\sum_{\mu\in E^*v}e^{-s|\mu|}=\sum_{n=0}^\infty |E^nv|e^{-s n}, \quad s\in\Cz.
\] 

\begin{remark}
When it will not cause confusion, we will drop the superscript and simply write $Z_v(s)$. 
The function $Z_v(s)$ is called the \emph{partition function with fixed-target $v$} in \cite{CarLar} (see \cite[Equation~5.8]{CarLar}). This terminology originates in \cite[Defintion~9.3]{ExLa} in the study of Cuntz--Krieger algebras.
\end{remark}

We will follow the notation from \cite{CarLar}, and let $\beta_v:=\inf(\{\beta\in \Rz : Z_v(\beta)<\infty\})$, so that $Z_v(\beta)<\infty$ for every $\beta>\beta_v$.
If $\beta\in\Rz^*$ is such that $Z_v(\beta)<\infty$, then we define the probability measure $m_{v,\beta}$ on $E^{\leq\infty}$ by
\[
m_{v,\beta}:=\frac{1}{Z_v(\beta)}\sum_{\mu\in E^*v}e^{-\beta |\mu|}\delta_\mu.
\]

In order to use our results from Section~\ref{sec:pfncs&gpoids}, we shall need the following, likely well-known, easy lemma.

\begin{lemma}
\label{lem:cFprin}
The groupoid $c^{-1}(0)$ is principal.
\end{lemma}

\begin{proof}
Let $x\in E^{\leq\infty}$. Any element of the isotropy group at $x$ is of the form $(x,n,x)$ for some $n\in \Zz$. Now $c((x,n,x))=0$ if and only if $n=0$, so $c^{-1}(0)\cap (\G_E)_x^x=\{(x,0,x)\}$, that is, $c^{-1}(0)$ is principal.
\end{proof}

The proof of the following lemma is routine (cf. Remark~\ref{rmk:consistent}), so we omit it.
\begin{lemma}
\label{lem:KMSv}
Let $v\in E^0$. For each $\beta>\beta_v$, 

\begin{enumerate}[\upshape(i)]
	\item the orbit $\G_Ev=E^*v$ is $\beta$-summable;
	\item $m_{\G_Ev}=m_{v,\beta}$, so that $m_{v,\beta}$ is $e^\beta$-conformal;
	\item $Z_{\G_Ev}(s)=Z_v(s)$.
\end{enumerate}
Moreover, the KMS$_\beta$ state on $\T C^*(E)$ associated with the $\beta$-summable orbit $E^*v$ via Proposition~\ref{prop:gpoidmodel} and equation \eqref{eqn:measurefromorbit} is given by $\phi_{v,\beta}:=m_{v,\beta}\circ\Phi$, where $\Phi$ is the canonical conditional expectation from $\T C^*(E)$ onto $C_0(E^{\leq\infty})$.
\end{lemma}

We will now prove that if $\beta$ is positive, then every $\beta$-summable orbit arises from a vertex, that is, is of the form $E^*v$ (see Theorem~\ref{thm:Toeplitz} below). We begin by making an observation about $\beta$-summability for orbits of infinite paths.

\begin{lemma}
\label{lem:xnotsummable}
Let $\beta\in(0,\infty)$. If $x=x_0x_1...\in E^\infty$, then the orbit $[x]:=\G_Ex\subseteq E^{\leq\infty}$ is not $\beta$-summable.
\end{lemma}

\begin{proof}
For each $n\geq 0$, let $x[n,\infty]:=x_nx_{n+1}...\in[x]$, let $x[0,n-1]:=x_0,...,x_{n-1}$, and let $\gamma_n:=(x[n,\infty],-n,x) \in\G_E$. Then we have $r(\gamma_n)=x[n,\infty]$, $s(\gamma_n)=x$, and $c(\gamma_n)=-n$. 
If $[x]$ were $\beta$-summable, then we would have
\[
\sum_{n=1}^\infty e^{-\beta c(\gamma_n)}\leq \sum_{\gamma\in [x]}l_z(\gamma)^\beta<\infty,
\]
but the sum on the left hand side of this equation diverges since $-\beta c(\gamma_n)=\beta n\geq 0$.
\end{proof}

For $\beta\in\Rz$, we shall follow the notation from \cite[Proposition~5.8]{CarLar} and let $E^0_{\beta\textup{-reg}}:=\{v\in E^0 : Z_v(\beta)<\infty\}$ be the set of \emph{$\beta$-regular vertices}. The parameterisation given by part (i) in the next result could also be derived from \cite[Theorem~5.6]{CarLar}; however, we give an alternative proof here using our general results from Section~ \ref{sec:summable}. The main advantage of our approach is that it allows us to compute all partition functions of the C*-dynamical system $(\T C^*(E),\alpha)$, which is one of our primary objectives. As a consequence of our computation, we see that the (abstract) partition functions of $(\T C^*(E),\alpha)$ in the sense of Definition~\ref{def:partitionfnc} (cf. \cite{BLT}) agree with the so-called ``fixed-target'' partition functions from \cite[Section~5]{CarLar}. We point out that the (abstract) notion of partition function from \cite{BLT} was not available when \cite{CarLar} was written.

\begin{theorem}
	\label{thm:Toeplitz}
	Let $E$ be a countable directed graph, and let $\alpha$ be the gauge action on $\T C^*(E)$.
	\begin{enumerate}[\upshape(i)]
		\item For each $\beta\in(0,\infty)$, the map
		\begin{equation}
		\label{eqn:verticestoKMS}
		E^0_{\beta\textup{-reg}}\to \E_{\rm I}(\KMS_\beta(\T C^*(E), \alpha)), \quad  v\mapsto \phi_{v,\beta},
		\end{equation}
		is a bijection  (cf. \cite[Theorem~5.6]{CarLar});
		\item for each $v\in E^0_{\beta\textup{-reg}}$, the partition function of $\phi_{v,\beta}$ is equal to $Z_v(s)$.
	\end{enumerate}
\end{theorem}

\begin{proof}
By Lemma~\ref{lem:cFprin}, $c^{-1}(0)$ is principal, so that Proposition~\ref{prop:main} gives us a bijection 
\[
\cO_\beta(\G_E,c)\to \E_{\rm I}(\KMS_\beta(\T C^*(E), \alpha)), \quad O\mapsto \phi_O.
\]
By Lemma~\ref{lem:xnotsummable}, no orbit coming from $E^\infty$ is $\beta$-summable. Each orbit in $E^*$ is of the form $E^*v$, and as we have already observed, such an orbit is $\beta$-summable if and only if $Z_v(\beta)<\infty$. Since $\varphi_{E^*v}=\varphi_{v,\beta}$ (see Lemma~\ref{lem:KMSv}), we see that the map given in \eqref{eqn:verticestoKMS} is a bijection.

The claim about partition functions follows from Lemma~\ref{lem:KMSv} combined with Theorem~\ref{thm:irredrep}.
\end{proof}

As an immediate consequence, we obtain:
\begin{corollary}
	\label{cor:Toeplitz}
	Let $E$ and $F$ be countable directed graphs, and let $\alpha^E$ and $\alpha^F$ denote the gauge actions on $\cT C^*(E)$ and $\cT C^*(F)$, respectively. If there is an isomorphism of C*-dynamical systems $(\cT C^*(E),\alpha^E)\cong (\cT C^*(F),\alpha^F)$, then, for each $\beta\in (0,\infty)$, there is a bijection $E^0_{\beta\textup{-reg}}\cong F^0_{\beta\textup{-reg}}$, $v\mapsto v'$, such that $Z_v^E(s)=Z_{v'}^F(s)$ for all $v\in E^0_{\beta\textup{-reg}}$.
\end{corollary}

\begin{remark}
The equality $Z_v^E(s)=Z_{v'}^F(s)$ is equivalent to $|E^nv|=|F^nv'|$ for every $n\geq 0$.
\end{remark}

Our analysis of $(\T C^*(E),\alpha)$ has the following implication for the C*-dynamical system $(C^*(E),\bar{\alpha})$:

\begin{theorem}
	\label{thm:graph}
	Let $E$ be a countable directed graph, and let $\bar{\alpha}$ denote the gauge action on $C^*(E)$.
	\begin{enumerate}[\upshape(i)]
		\item For each $\beta\in(0,\infty)$, the map
		\[
		E^0_{\beta\textup{-reg}}\setminus E_{\rm reg}^0\to \E_{\rm I}(\KMS_\beta(C^*(E), \bar{\alpha})), \quad \ v\mapsto \bar{\phi}_{v,\beta},
		\]
		is a bijection, where $\bar{\phi}_{v,\beta}$ is the KMS$_\beta$ state on $C^*(E)$ associated with $m_{v,\beta}$;
		\item for each $v\in E^0_{\beta\textup{-reg}}\setminus E_{\rm reg}^0$, the partition function of $\bar{\phi}_{v,\beta}$ is equal to $Z_v(s)$.
	\end{enumerate}
\end{theorem}
\begin{proof}
	The proof is almost identical to that of Theorem~\ref{thm:Toeplitz}.
\end{proof}

Our next result shows that when $E^0_{\beta\textup{-reg}}$ coincides with $E^0$, and the function $v\mapsto Z_v(\beta)$ is in $\ell^\infty(E^0)$, then all extremal  KMS$_\beta$ states come from vertices in $E$. 

\begin{theorem}
	\label{thm:KMSwithgrowthcond}
	Let $E$ be a countable directed graph, let $\alpha$ be the gauge action on $\T C^*(E)$, and fix $\beta\in (0,\infty)$. If $\sup_{v\in E^0}Z_v(\beta)<\infty$, then every extremal KMS$_{\beta}$ state on $\T C^*(E)$ is of type I, and the map
	\[
	E^0\to \E(\KMS_{\beta}(\T C^*(E), \alpha)), \quad \ v\mapsto \phi_{v,\beta},
	\]
	is a bijection. 
\end{theorem}

\begin{proof}
	Suppose that $\phi$ is an extremal KMS$_\beta$ state on $\T C^*(E)$, and let $m$ be the $e^\beta$-conformal probability measure on $E^{\leq \infty}$ associated with $\phi$. Then $m$ is determined by the values $m(\cZ(\mu))$ for $\mu\in E^*$, and
	\[
	m(\cZ(\mu))=\phi(s_\mu s_\mu^*)=e^{-\beta |\mu|}\phi(s_\mu^*s_\mu)=e^{-\beta |\mu|}m(\cZ(r(\mu))),
	\]
	where the second equality uses the KMS$_\beta$ condition. Since $\sum_{v\in E^0}m(\cZ(v))=1$, we have
	\[
	\sum_{\mu\in E^*}m(\cZ(\mu))=\sum_{v\in E^0}m(\cZ(v))\sum_{\mu\in E^*v}e^{-\beta |\mu|}=\sum_{v\in E^0}m(\cZ(v))Z_v(\beta)\leq\sup_{v\in E^0}Z_v(\beta)<\infty.
	\]
	Hence, the Borel--Cantelli lemma implies that $m$ is concentrated on the set 
	\[
	\{x\in E^{\leq \infty} : x\in \cZ(\mu) \text{ for only finitely many } \mu\in E^*\}=E^*.
	\]
	Since $m$ is extremal in $\QI_\beta(\G_E,c)$, it follows that $m$ must be concentrated on an orbit in $E^*$, that is, on a set of the form $E^*v$ for some $v\in E^0$. Therefore, $m=m_{v,\beta}$ for some vertex $v\in E^0$. By Proposition~\ref{prop:KumRen}, we have $\phi=\phi_{v,\beta}$.
\end{proof}

We demonstrate in the following example that having $E^0=E^0_{\beta\textup{-reg}}$ is not sufficient to conclude all KMS$_\beta$ are of type I; one needs the stronger assumption $\sup_{v\in E^0}Z_v(\beta)<\infty$.

\begin{example}
\label{exm:CarLarExample}
Consider the graph $E$ with vertex set $E^0=\{v_n : n\geq 0\}$, edge set $E^1=\{e_n :n\geq 0\}\cup\{f_n : n\geq 0\}$ with $r(e_n)=s(e_n)=v_n$, $s(f_n)=v_n$, and $r(f_n)=v_{n+1}$ for all $n\geq 0$; $E$ is depicted by 
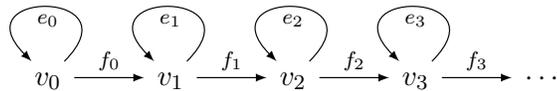
\begin{equation}\label{pic:CarLarGraph}
\begin{tikzcd}[
arrow style=tikz,
>={latex} 
]
v_0\arrow[loop,"e_0"]\arrow[r,"f_0"] & v_1\arrow[loop,"e_1"]\arrow[r,"f_1"] & v_2\arrow[loop,"e_2"]\arrow[r,"f_2"] & v_3 \arrow[loop,"e_3"]\arrow[r,"f_3"] &\cdots 
\end{tikzcd}
\end{equation}
The KMS states on the Toeplitz algebra of $E$ have been parameterised in \cite[Theorem~7.5]{CarLar}. 
As a consequence, we have $E^0=E^0_{\beta\textup{-reg}}$ for $\beta\in (0,\infty)$, and there is a unique $e^\beta$-conformal probability measure $m_{\inf}^\beta$ concentrated on $E^{\infty}$ for $\beta\in (0,\log 2)$, which is not of type I by Lemma~\ref{lem:xnotsummable} and Theorem~\ref{thm:Toeplitz}. 

Here, we compute the partition functions of the C*-dynamical system $(\T C^*(E),\alpha)$. For $\beta\in (0,\infty)$, we have 
\begin{equation}
Z_{v_n}(\beta)=\begin{cases}
2(n+1) & \text{ if } \beta=\log 2,\\
\frac{1}{1-e^{-\beta}}\left(\frac{1-a^{n+1}}{1-a}\right) & \text{ if } \beta\neq \log 2,
\end{cases}
\end{equation}
for each $n\geq 0$ where $a:=\frac{e^{-\beta}}{1-e^{-\beta}}$. This follows from the recurrence relation 
\begin{equation}
Z_{v_{n+1}}(\beta) = \sum_{\mu \in v_{n+1}E^*v_{n+1}} e^{-\beta |\mu|} + \sum_{k=0}^\infty \sum_{\nu \in E^*v_{n}} e^{-\beta |\nu f_n e_{n+1}^k|} 
= \frac{1}{1-e^{-\beta}} + a Z_{v_n}(\beta) 
\end{equation}
of partition functions, where $n \in\Nz$ and $\beta \in (0,\infty)$, combined with 
\begin{equation}
Z_{v_0}(\beta) = \sum_{\mu \in v_0E^*v_0} e^{-\beta |\mu|} = \frac{1}{1-e^{-\beta}}. 
\end{equation} 
In particular, $\sup_{n} Z_{v_n}(\beta)$ is finite if and only if $\beta\in (\log 2,\infty)$. Note that it is possible to apply results from \cite[Example~7.4]{CarLar} to compute these partition functions. 

For $\beta>0$ and $n\geq 0$, we have $Z_{v_n}(\beta)<\infty$, that is, $E^0=E^0_{\beta\textup{-reg}}$ for every $\beta>0$. Hence, Theorem~\ref{thm:graph} implies that for every $\beta>0$, the set of extremal KMS$_\beta$ states that are of type I is precisely $\{\varphi_{v_n,\beta} : n\geq 0\}$. For $\beta\in (\log 2,\infty)$, we have $\sup_{n} Z_{v_n}(\beta)<\infty$, so that Theorem~\ref{thm:KMSwithgrowthcond} implies that every KMS$_\beta$ state is of type I. Needless to say, both these observations are consistent with \cite[Theorem~7.5]{CarLar}. 
For $\beta\in(0,\log 2)$, we have $\sup_{n} Z_{v_n}(\beta)=\infty$, so the hypotheses of Theorem~\ref{thm:KMSwithgrowthcond} are not satisfied. In this case, the KMS$_\beta$ state $\psi_\beta$ corresponding to the measure $m_{\inf}^\beta$ is not of type I. This demonstrates that the assumption that every vertex is $\beta$-regular is not sufficient to guarantee that every KMS$_\beta$ state is of type I. We point out that this phenomenon can only occur when the graph has infinitely many vertices.
\end{example}

\subsection{Admissible triples and ground states}

\label{subsec:triples}

It follows from Corollary~\ref{cor:Toeplitz} that we can recover information from $(\T C^*(E),\alpha)$ about the number of finite paths in $E$ that terminate at a vertex in $\bigcup_{\beta>0}E^0_{\beta\textup{-reg}}$. However, it may well happen that few vertices are $\beta$-regular. For example, the graph with a single vertex and a countably infinite set of edges,  which gives rise to $\cO_\infty$, has no $\beta$-regular vertices (cf. \cite[Example~7.10]{CarLar}).
Ground states exist regardless of any $\beta$-summability considerations, so by considering admissible triples of ground states we are able to recover some interesting information about finite paths in general, even when $(\T C^*(E),\alpha)$ has no KMS$_\beta$ states for every $\beta<\infty$.

The KMS$_\infty$ and ground states of $(\T C^*(E),\alpha)$ have been computed in \cite{CarLar}; we shall only need the explicit description of all extremal ground states. For each vertex $v\in E^0$, let $m_v$ denote the point-mass measure at $v\in E^*\subseteq E^{\leq \infty}$. Then $\phi_v:=m_v\circ\Phi$ is a state on $\T C^*(E)$ such that 
\[
\phi_v(s_\mu s_\nu^*)=\begin{cases}
1 & \text{ if } \mu=\nu=v,\\
0 & \text{ otherwise,}
\end{cases}
\]
for $\mu,\nu\in E^*$ with $r(\mu)=r(\nu)$.

\begin{theorem}[{\cite[Theorem~4.1~Section~\S~6]{CarLar}}]
	\label{thm:ground}
	We have $\E(\Gr(\T C^*(E),\alpha))=\{\varphi_v : v\in E^0\}$.
\end{theorem}

\begin{remark}
    Given $\beta,\tilde{\beta}\in [0,\infty)$ such that $\beta\leq \tilde{\beta}$, we see that $Z_v(\beta)$ converges, then $Z_v(\tilde{\beta})$ must also converge. Thus, we have the inclusion $E^0_{\beta\textup{-reg}}\subseteq E^0_{\tilde{\beta}\textup{-reg}}$ whenever  $\beta\leq \tilde{\beta}$. If $v\in \bigcup_{\beta>0}E^0_{\beta\textup{-reg}}$, then the weak* limit $\lim_{\tilde{\beta}>\beta}\varphi_{v,\tilde{\beta}}$ exists and coincides with $\varphi_v$ (see \cite[Corollary~2.7]{BLT}). Thus, each such $\varphi_v$ is a KMS$_\infty$ state, and we obtain an injection of $\bigcup_{\beta>0}E^0_{\beta\textup{-reg}}$ into $\E_{\rm I}(\Gr(\T C^*(E),\alpha))$ (cf. \cite[Theorem~6.1]{CarLar}). For the definition of KMS$_\infty$ state, see \cite[Definition~3.7]{CM}.
\end{remark}

For each $v\in E^0$, let $\pi_v\colon \T C^*(E)\to\cB(\ell^2(E^*v))$ be the irreducible representation such that for $w\in E^0$, $\pi_v(p_w)$ is the orthogonal projection onto $\ell^2(wE^*v)$, and
\[
\pi_v(s_\mu)\delta_\nu=\begin{cases}
\delta_{\mu\nu} & \text{ if } r(\mu)=s(\nu),\\
0 & \text{ if }r(\mu)\neq s(\nu),
\end{cases}
\]
for all $\mu\in E^*$ and $\nu\in E^*v$ (cf. Lemma~\ref{lem:SW}). Here, $\{\delta_\nu : \nu\in E^*v\}$ is the canonical orthonormal basis for $\ell^2(E^*v)$ consisting of point-mass functions. Note that when $v$ is $\beta$-regular, $\pi_v$ coincides with the representation in the admissible triple for $\varphi_{v,\beta}$. 
Let $\{U_t\}_{t\in\Rz}$ be the one-parameter unitary group on $\ell^2(E^*v)$ given by $U_t\delta_\mu=e^{it|\mu|}\delta_\mu$ for $\mu\in E^*$ and  $t\in\Rz$, and let $H_v={\rm diag}(|\mu|)_{\mu\in E^*v}$ be the unique positive generator of the unitary group $\{U_t\}_{t\in\Rz}$ that has $0$ in its spectrum.
\begin{proposition}
	\label{prop:groundGNS}
	If $v\in E^0$, then $(\pi_v,\ell^2(E^*v),e^{-H_v})$ is an admissible triple for $\phi_v$.
\end{proposition}

\begin{proof}
We can see that $\varphi_v(a)=\langle \pi_v(a)\delta_v, \delta_v\rangle$ and $\delta_v$ is a cyclic vector, so that $(\pi_v,\ell^2(E^*v), \delta_v)$ is the GNS-representation of $\varphi$. In addition, a calculation shows that $(U_t \pi_v(s_{\mu})U_{-t})\delta_{\nu} = e^{it|\mu|}\pi(s_\mu)\delta_{\nu}= \pi_v(\sigma_t(s_\mu))\delta_\nu$ for any $\mu,\nu \in E^*$. Hence  $(\pi_v,\ell^2(E^*v),e^{-H_v})$ is an admissible triple.
\end{proof}

We are now ready for the main result of this subsection.
\begin{theorem}
	\label{thm:general}
Let $E$ and $F$ be countable directed graphs. If there exists an isomorphism of C*-dynamical systems $(\T C^*(E),\alpha^E)\cong (\T C^*(F),\alpha^F)$, then there is a bijection $E^0\to F^0$, $v\mapsto v'$, such that $|E^nv|=|F^nv'|$ for every $v\in E^0$ and every $n\geq 0$.
\end{theorem}

\begin{proof}
Suppose there is an isomorphism $(\T C^*(E),\alpha^E)\cong (\T C^*(F),\alpha^F)$. Then by Theorem~\ref{thm:ground}, we have bijections $E^0\to \E_{\rm I}(\Gr(\T C^*(E),\alpha^E))$, $v\mapsto \phi_v^E$, and $F^0\to \E_{\rm I}(\Gr(\T C^*(F),\alpha^F))$, $w\mapsto \phi_w^F$, so that we obtain a bijection $E^0\cong F^0$, $v\mapsto v'$, such that $\phi_v^E\mapsto\phi_{v'}^F$. 

By Proposition~\ref{prop:groundGNS}, $(\pi_v,\ell^2(E^*v),e^{-H_v})$ is an admissible triple for $\phi_v$, so all spectral information contained in $H_v$ can be recovered from $(\T C^*(E),\alpha)$. Let $e^{-H_v}=\textup{diag}(e^{-|\mu|})_{\mu\in E^*v}$ be the (bounded) diagonal operator on $\ell^2(E^*v)$ such that, for every $\mu\in E^*v$, $e^{-H_v}\delta_\mu=e^{-|\mu|}\delta_\mu$, and let $\ell^2(E^*v)_n$ denote the eigenspace for $e^{-H_v}$ corresponding to the eigenvalue $e^{-n}$.
The basis vector $\delta_\mu$ has eigenvalue $e^{-n}$ if and only if $|\mu|=n$, that is, $\mu\in E^n$. Thus, the $n$-th eigenspace for $e^{-H_v}$ is $\ell^2(E^nv)$. Similar observations hold for the graph $F$, so we have $|E^nv|=\dim\ell^2(E^nv)=\dim\ell^2(F^nv')=|F^nv'|$ for every $v\in E^0$ and every $n\geq 0$.
\end{proof}

\subsubsection{Graph reconstruction from vertex algebra preserving isomorphisms of C*-dynamical systems}

We now show that our results can be used to obtain a generalisation of one of the main results from \cite{BLRS} to the case of infinite graphs. This generalisation can also be proved using entirely different techniques without mention of KMS or ground states, see \cite[Section~6]{DEG}, but we emphasise that the approach using ground states allows one to recover the graph explicitly in terms of equilibrium states of the C*-dynamical system. Given a countable directed graph $E$, we follow \cite{BLRS} and consider the vertex algebra $M_E:=\overline{\spn}(\{p_w : w\in E^0\})\subseteq \T C^*(E)$. The following theorem shows that one can recover $E$ explicitly from the triple $(\T C^*(E),\alpha,M_E)$.

\begin{theorem}
	\label{thm:recon}
Let $E$ be an arbitrary countable directed graph, and let $\alpha$ denote the gauge action on $\T C^*(E)$. Let $\tilde{E}$ be the directed graph with vertex set $\tilde{E}^0:=\E(\Gr(\T C^*(E),\alpha))$ and edge set defined by
\[
|\phi_w \tilde{E}^1\phi_v|:=\dim(\pi_v(p_w)\ell^2(E^1v)).
\]
Then 

\begin{enumerate}[\upshape(i)]
	\item $\tilde{E}$ is an isomorphism invariant of the triple $(\T C^*(E),\alpha,M_E)$;
	\item $E\cong\tilde{E}$.
\end{enumerate}
\end{theorem}

Before continuing to the proof of Theorem~\ref{thm:recon}, we state an immediate corollary, which generalises \cite[Theorem~3.2(1)]{BLRS} to the case of infinite graphs. 

\begin{corollary}[cf. {\cite[Theorem~6.1]{DEG}}]
	\label{cor:recon}
Let $E$ and $F$ be countable directed graphs. Then there is an isomorphism of C*-dynamical systems $(\T C^*(E),\alpha^E)\cong (\T C^*(F),\alpha^F)$ that carries $M_E$ onto $M_F$ if and only if $E\cong F$.
\end{corollary}

Our approach to proving Theorem~\ref{thm:recon} will be somewhat different than the proof of \cite[Theorem~3.2(1)]{BLRS}; this is necessitated by the fact that for infinite graphs, there may be no KMS$_\beta$ states at all. 

\begin{lemma}
	\label{lem:ME}
	 For each vertex $v\in E^0$, $p_v$ is the unique non-zero minimal projection $p\in M_E$ such that $\phi_v(p)=1$.
\end{lemma}

\begin{proof}
	Since $\phi_v(\cdot)=\langle\pi_v(\cdot)\delta_v,\delta_v\rangle$, we see that $\phi_v(p_v)=1$.
	Suppose that $p\in M_E$ is a minimal projection such that $\phi_v(p)=1$. Write $p=\lim_n\sum_{w\in E^0} z_w^np_{w}$ where $z_w^n\in\Cz$ are such that $\{w : z_w^n\neq 0\}$ is finite for every $n$. 
	We have $1=\phi_v\left(\lim_n\sum_{w\in E^0} z_w^np_{w}\right)=\lim_nz_v^n$,
	and thus $p_vp=\lim_n\sum_{w\in E^0} z_w^np_vp_w=\lim_nz_v^np_v=p_v$.
	That is, $p_v\leq p$. Since $p$ is minimal, this forces $p=p_v$.
\end{proof}

\begin{proof}[Proof of Theorem~\ref{thm:recon}]
We shall continue with the notation from the proof of Theorem~\ref{thm:general}. 

Proof of (i): By Theorem~\ref{thm:ground}, the map $v\mapsto \varphi_v$ is a bijection from $E^0$ onto $\E(\Gr(\T C^*(E), \alpha))$, so only need to show that we can read off the numbers $|\phi_w\tilde{E}^1\phi_v|=\dim(\pi_v(p_w)\ell^2(E^*v)_1)$ from $(\T C^*(E),\alpha,M_E)$.
Since $\pi_v(p_w)$ is the orthogonal projection from $\ell^2(E^*v)$ onto the subspace $\ell^2(wE^*v)$, we see that $\pi_v(p_w)\ell^2(E^*v)_n=\ell^2(wE^nv)$ for every $w\in E^0$ and $n\geq 0$. Thus, by Proposition~\ref{prop:groundGNS} and Lemma~\ref{lem:ME}, we can read off the numbers $\dim(\pi_v(p_w)\ell^2(E^1v)_1)$.
	
Proof of (ii): The bijection $E^0\to \tilde{E}^0$ given by $v\mapsto\phi_v$ defines a graph isomorphism since $|wE^1v|=\dim(\ell^2(wE^1v))=\dim(\pi_v(p_w)\ell^2(E^1v)_1)=|\phi_w\tilde{E}^1\phi_v|$.
\end{proof}

\subsection{A geometric interpretation of the critical inverse temperatures} 
\label{ssec:geom}
As we have seen in Theorem \ref{thm:general}, $|E^nv|$ is a C*-dynamical invariant for any $n \in \Nz$ and $v \in E^0$ even when the partition function $Z_v(\beta)$ does not converge. However, as we shall see in this subsection, we can recover another invariant of graphs from the convergence condition of partition functions. 

As before, for each $v\in E^0$, we let $\beta_v=\inf(\{\beta\in \Rz : Z_v(\beta)<\infty\})$ denote the abscissa of convergence of $Z_v(\beta)$. We shall refer to $\beta_v\in \Rz \cup \{\pm \infty\}$ as the \emph{critical inverse temperatures} of the C*-dynamical system $(\T C^*(E),\alpha)$.

\subsubsection{Summary of Hausdorff dimension}
First, we summarise some known results about Hausdorff dimension of trees. Throughout this subsection, $T$ denotes a locally finite infinite rooted tree with root $o$. We shall consider such trees as directed graphs. For $v \in T^0$, let $|v|$ be the distance of $v$ from $o$. If there is an edge $e$ between $v_1$ and $v_2$ for $|v_1|+1 = |v_2|$, then $s(e) = v_1$ and $r(e)=v_2$. As usual, we tacitly assume that any vertex in $T$ can be reached from the root. 
\begin{definition}[cf. {\cite[p.12]{LyoPer}}] \label{def:Hausdorff}
	Let $\partial T$ be the compact space of infinite paths which start from  $o$. 
	The metric on $\partial T$ is defined by 
	\[
	d(x,y) =\inf\{e^{-n} : x_i=y_i \text{ for } 1 \leq i \leq n \text{ and } x_{n+1}\neq y_{n+1} \}\quad \text{for }x,y\in \partial T.
	\]
	A collection $\cC$ of subsets of $\partial T$ is a \emph{cover} if $\bigcup \cC = \partial T$. The \emph{Hausdorff dimension of $\partial T$} is defined by 
	\[ \dim \partial T = \sup \{ \beta >0 \colon \inf \{ \sum_{B \in \cC} (\textup{diam\,} B)^\beta \colon \cC \mbox{ is a countable cover of $\partial T$}\} >0\}. \]
\end{definition}

\begin{remark}
	    Note that $\partial T$ is the boundary of the rooted tree $T$, which is not the same as the boundary of $T$ when viewed as a directed graph.
	\end{remark}

\begin{definition}[cf. {\cite[p.81]{LyoPer}}]
Let $T$ be a locally finite infinite rooted tree. 
The \emph{upper (resp. lower) growth rate of $T$} is defined by 
	\begin{align*}
		\ugr T &= \limsup_{n \to \infty} |\{v \in T \colon |v|=n\}|^{1/n} \\
		(\mbox{resp. }\lgr T &= \liminf_{n \to \infty} |\{v \in T \colon |v|=n\}|^{1/n}). 
	\end{align*}
\end{definition}

\begin{definition}[cf. {\cite[p.82]{LyoPer}}]
	For $v \in T$, let $T_v$ be the subtree of $T$ consisting of all descendants of $v$ and whose root is $v$. 
	A tree $T$ is \emph{periodic} if there exists a constant $N \in \Nz$ such that for any vertex $v \in T^0$, there exists a vertex $w \in T^0$ with $|w| \leq N$ such that $T_v$ is isomorphic to $T_w$ as rooted trees. 
\end{definition}

For the Hausdorff dimensions of periodic trees, the following formula is known.

\begin{lemma} \label{lem:minkowski}
	If $T$ is periodic, then we have 
	\[ \dim \partial T = \log \ugr T = \log \lgr T . \]
\end{lemma}

\begin{proof}
	This follows from \cite[Theorem 3.8]{LyoPer} combined with the fact
	\[ \dim \partial T = \log \br T, \]
	where $\br T$ is the branching number of $T$ (see \cite[Section 3]{LyoPer}). 
\end{proof}

\subsubsection{The connection with critical inverse temperatures}

Recall that the opposite graph of a directed graph $E$ is the directed graph $E^{\rm op}$ obtained from $E$ by reversing the direction of every edge.

\begin{definition}[cf. {\cite[p.82]{LyoPer}}]
\label{def:dircover}
	For a directed graph $E$ and $v \in E^0$, let $T(E,v)$ denote the \emph{directed cover of $E^{\mathrm{op}}$ based at $v$}. That is, $T(E,v)$ is the rooted tree defined as follows:

	\begin{itemize} 
		\item The vertex set of $T(E,v)$ is $E^*v$, the set of all finite paths in $E$ which end at $v$. 
		\item The root of $T(E,v)$ is $v$ (it is the path of length $0$). 
		\item For $x,y \in T(E,v)^0 = E^*v$, there is an edge $e \in T(E,v)^1$ with $s(e)=x, r(e)=y$ if and only if $y=ex$ for some edge $e \in E^1$. 
	\end{itemize}
\end{definition}

Note that $T(E,v)$ is locally finite if $v \in E^0$ can not be reached from any vertex $w$ that receives infinitely many edges.

If $E$ is a finite graph, then $T(E,v)$ is a periodic tree. Indeed, for each vertex $w$ of $E$ such that there is a path from $w$ to $v$, take $\mu_{w}$ to be any such path. Then, for any $x \in T(E,v)^0$, $T(E,v)_x$ is isomorphic to $T(E,v)_{\mu_{s(x)}}$, which means that the periodicity is satisfied by the constant $N= \max_{w} |\mu_w|$. 

\begin{proposition}
	\label{prop:betac=dim}
Let $E$ be a directed graph and let $v \in E^0$ with $\beta_v<\infty$. 
	\begin{enumerate}[\upshape(i)]
		\item We have  $\beta_v = \log \ugr T(E,v)$. 
		\item If $E$ is a finite graph, then we have 
		$\beta_v = \dim \partial T(E,v)$.
	\end{enumerate}
\end{proposition}

\begin{proof}
	Let $T=T(E,v)$. The convergence condition of $Z_v(\beta)$ for $\beta>\beta_v$ implies that the rooted tree  $T$ is locally finite. 
	Let 
	\[ a_n := |\{v \in T \colon |v|=n\}|, \]
	and let $R$ be the radius of convergence of the power series $f(z) := \sum_{n=0}^\infty a_n z^n$. By Hadamard's formula, we have 
	\[ \frac1R = \limsup_{n \to \infty} a_n^{1/n} = \ugr T. \]
	Since $Z_v(\beta) = f(e^{-\beta})$, $Z_v(\beta)$ converges if $-\beta < \log R$ and diverges if $-\beta > \log R$. Hence, 
	\[ \beta_v = \log \frac1R = \log \ugr T. \]
	The latter half of the claim follows from Lemma~\ref{lem:minkowski}. 
\end{proof}

As we have seen in Proposition \ref{prop:betac=dim}, if $E$ is a finite graph, the abscissa of convergence of $Z_v(s)$ for any $v \in E^0$ is equal to the dimension of $\partial T(E,v)$. For infinite graphs, in fact, the critical inverse temperature is related to another dimension. 
\begin{definition} \cite[p.512]{LyoPer}
    Let $T$ be a locally finite infinite tree. Then, \emph{the upper Minkowski dimension of $\partial T$} is defined by
    \[ \dim^M \partial T = \limsup_{\varepsilon \to +0} \frac{\log N(\partial T, \varepsilon)}{- \log \varepsilon}, \]
    where $N(\partial T, \varepsilon)$ is the minimum cardinality of covers consisting of balls $B$ with $ \textup{diam\,} B \leq \varepsilon$. Here, we consider the metric on $\partial T$ from Definition \ref{def:Hausdorff}. 
\end{definition}
Let $T$ be a locally finite infinite tree, and let $T'$ be a tree obtained by removing all vertices $v \in T^0$ such that there are no infinite paths which go through $v$. If 
\begin{equation}
   \label{eqn:Minkowski}
   \ugr T = \ugr T', 
\end{equation}
then we can see that $\log \ugr T = \dim^M \partial T.$
The condition \eqref{eqn:Minkowski} says that $T$ does not have ``too many sinks". Hence, if a graph $E$ does not have ``too many sources" connecting to $v \in E^0$, then we have 
\[ \beta_v = \dim^M \partial T(E,v) \]
under the same assumption as Proposition \ref{prop:betac=dim}. 
Note that $\dim T = \dim^M T$ if $T$ is periodic by Lemma \ref{lem:minkowski}. 
\begin{example}
    We give two sufficient conditions to have the condition \eqref{eqn:Minkowski}. Let $T$ be a locally finite infinite tree. Then, if we have
    \[ \limsup_{n \to \infty} \frac{|\{v \in T^0 \mid |v|=n \mbox{ and there are no infinite paths going through $v$} \}|}{|\{v \in T^0 \mid |v|=n \}|} < 1, \]
    then the condition \eqref{eqn:Minkowski} is satisfied. Moreover, if we have 
    \[ \sup_{n \in \Nz} |\{v \in T^0 \mid |v|=n \mbox{ and $v$ is a sink}\}| < \infty, \]
    then the condition \eqref{eqn:Minkowski} is also satisfied. 
\end{example}

\section{Type III KMS states of finite graph algebras}
\label{sec:finitegraphs&type}

We have visited KMS$_\beta$ states for $\beta>\beta_v$, and we geometrically described the critical inverse temperatures $\beta_v$. These KMS states are of type I, which enables us to define partition functions. 
In this section, we observe the behaviour of KMS states at the critical inverse temperature for finite graphs. KMS states on Toeplitz algebras of finite graphs have been carefully studied in \cite{aHLRS13,aHLRS15}. These efforts culminate in \cite[Theorem~5.3]{aHLRS15} which provides a detailed description of all KMS$_\beta$ states of $(\T C^*(E),\alpha)$, including a description of which KMS states factor through quotients of $\T C^*(E)$.
The main result of this section is Theorem \ref{thm:typedetermination}, which is a generalisation of type computations for C*-algebras of strongly connected finite  graphs equipped with the gauge action (see, for instance, \cite[Theorem 3.1(2)(b)]{LLNSW} or \cite[Theorem 5]{Th19}). 
Consequently, KMS states at critical inverse temperatures are of type III$_\lambda$, and the $\lambda$-parameter is related to the value of critical inverse temperatures. This completes the analysis from \cite{aHLRS13,aHLRS15} of the KMS structure of the C*-dynamical system $(\T C^*(E),\alpha)$ in the case where $E$ is a finite graph.

For any finite graph $E$ and $v \in E^0$, we have $\beta_v \in \{-\infty\} \cup [0,\infty)$. If $\beta_v = -\infty$, then the admissible triple corresponding to $\varphi_{\beta,v}$ for some $\beta$ is finite-dimensional. If $\beta_v = 0$, then KMS states at $\beta_v$ are tracial. Hence, we restrict our attention to the case $\beta_v \in (0, \infty)$. 

\subsection{Preliminaries on finite graphs}

We begin this section by preparing notations and establishing fundamental lemmas. Although the main interest in this section is finite graphs, most of definitions and lemmas are formulated for arbitrary countable graphs. 

For a countable graph $E$, and for a subset $S$ of $E^0$,
\begin{itemize}
\item $E_S$ denotes the restriction of $E$ to $S$. That is, 
\[ E^0_S=S,\ E^1_S = \{ e \in E^1 \colon s(e), r(e) \in S\}. \] 
Note that $E^0_S$ and $E^1_S$ can be empty.
\item The adjacency matrix of $E$ is denoted by $A_E$, and the adjacency matrix of $E_S$ is denoted by $A_S$. Note that adjacency matrices are of infinite size if $E^0$ is not finite. 
\end{itemize}

In addition, for a finite graph $E$ and a subset $S$ of $E^0$, 
\begin{itemize}
\item Define a projection $p_S \in \T C^*(E)$ by $p_S := \sum_{v \in S} p_v$. 
\item Define $\cZ(S) := \bigcup_{w \in S} \cZ(w)$ and $\cZ'(S) := \bigcup_{w \in S} \cZ'(w)$. 
\end{itemize}

\begin{definition}
Let $E$ be a countable graph. 
A subset $S$ of $E^0$ is called \emph{hereditary} if for any $v \in S$ and $w \in E^0$, $vE^*w \neq \emptyset$ implies $w \in S$. For a subset $S$ of $E^0$, the \emph{hereditary closure} of $S$ is the smallest hereditary subset of $E^0$ containing $S$. The hereditary closure of $S$ is denoted by $\bar{S}$. Concretely, we have 
\[ \bar{S} = \{ w \in E^0 \mid SE^*w \neq \emptyset \}.  \]
\end{definition}

\begin{definition}
Let $E$ be a countable graph. 
Define an equivalence relation $\sim$ on $E^0$ by 
\[ v \sim w \Longleftrightarrow vE^*w \neq \emptyset \mbox{ and } wE^*v \neq \emptyset. \]
An equivalence class $C \subseteq E^0$ with respect to $\sim$ is called a \emph{strongly connected component} of $E$. The set of all strongly connected components of $E$ is denoted by $\pi(E)$. 
\end{definition}

In this section, we consider $\pi(E)$ for a possibly empty graph $E$ (that is, $E^0=\emptyset$). If $E$ is an empty graph, then $\pi(E)=\emptyset$. Moreover,
the restriction of a graph $E$ to $S$ is always taken when $S$ is hereditary or the complement of $S$ is hereditary. In such cases, the next lemma is often applied. 

\begin{lemma} \label{lem:restriction}
Let $E$ be a countable graph and let $S \subseteq E^0$. Suppose that $S$ is hereditary or $E^0\setminus S$ is hereditary. Then, we have 
\[ \pi(E_S) \subseteq \pi(E), \]
which means that any strongly connected component in $E_S$ is a strongly connected component in $E$. 
\end{lemma}
\begin{proof}
Suppose that $S$ is hereditary, and suppose $C \in \pi(E_S)$. Since all vertices in $C$ are equivalent in $E$, there exists $C' \in \pi(E)$ such that $C \subseteq C'$. 
If $C \subsetneq C'$, then there exists $v \in E^0 \setminus S$ such that $C E^*v \neq \emptyset$, which contradicts the assumption that $S$ is a hereditary subset. 
Hence we have $C=C' \in \pi(E)$. The proof of the case that $E^0\setminus S$ is hereditary is similar. 
\end{proof}

Lemma \ref{lem:restriction} will often be applied to the subset appearing in the next lemma. 
\begin{lemma} \label{lem:hereditary}
Let $E$ be a countable graph. For any $v \in E^0$, the subset $E^0\setminus s(E^*v)$ is hereditary. 
\end{lemma}

\begin{proof}
Let $S = E^0\setminus s(E^*v)$, and let $w \in S$. Let $w' \in E^0$ be such that $wE^*w' \neq \emptyset$. If $w' \in s(E^*v)$, then we would have $wE^*v \neq \emptyset$ and hence $w \in s(E^*v)$, which is a contradiction. Hence, we have $w' \in S$, which implies that $S$ is hereditary. 
\end{proof}

Next, we recall the notion of harmonic vector that has been introduced by Thomsen \cite{Th17Adv}, which enables us to reduce problems of KMS states on $C^*(E)$ to problems in linear algebra. This idea is traditionally used in the study of KMS states of graph algebras, especially in order to apply the Perron--Frobenius theorem (see \cite{EFW} or \cite[Proposition 4.1]{aHLRS13}). 

\begin{definition}[{\cite[Lemma 2.6]{Th17Adv}}]\label{def:harmonicvector}
Let $\beta \in \Rz$, and 
let $E$ be a countable graph and let $A_E := [a_{vw}]_{v,w \in E^0}$ be its adjacency matrix. 
A nonzero vector $\harm = [\harm_v]_{v \in E^0}$ is called a \emph{$\beta$-harmonic vector} if it satisfies 
\[ \sum_w a_{vw}\harm_w = e^\beta \harm_v \mbox{ for every $v \in E^0$}. \]
\end{definition}
In particular, $\beta$-harmonic vectors are the same as nonnegative eigenvectors of $A_E$ with eigenvalue $e^\beta$. A $\beta$-harmonic vector $\harm$ is said to be \emph{normalized} if $\sum_v \harm_v = 1$. 
Note that we only consider the gauge action, whereas Thomsen considers more general time evolutions. 

We shall use the following form of a result by Thomsen:
\begin{theorem}[{\cite[Theorem 2.7]{Th17Adv}}] \label{thm:Thomsen2}
Let $\beta \in \Rz$. Then, there is a bijection from the set of $e^\beta$-conformal probability measures concentrated in $E^\infty$ to the set of normalised $\beta$-harmonic vectors defined by \[ m \mapsto \harm_m := [ m(\cZ'(v)) ]_{v \in E^0}.\]
\end{theorem}
\begin{proof}
The restriction of the map in \cite[Theorem 2.7]{Th17Adv} yields the bijection. Note that the last line of \cite[Lemma 2.6]{Th17Adv} is applied here. 
\end{proof}

We say that a KMS$_\beta$ state $\varphi$ of $\T C^*(E)$ or $C^*(E)$ is \emph{supported in $E^\infty$} if the $e^\beta$-conformal measure associated to $\varphi$ is concentrated in $E^\infty$. Note that for $\beta>0$, an extremal KMS$_\beta$ state $\varphi$ that is supported in $E^\infty$ cannot be of type I by Theorem~\ref{thm:Toeplitz}.

\subsection{Description of critical inverse temperatures}

We have observed a description of critical inverse temperatures in Section \ref{ssec:geom} for arbitrary countable graphs. In this subsection, we give a finer analysis of critical inverse temperatures in the case of finite graphs. For finite graphs, there is another description of critical inverse temperatures in \cite[Corollary 5.8, Lemma 7.2]{aHLRS15}. 
The aim of this subsection is to clarify the relation between our approach and \cite{aHLRS15}. Concretely, we give an alternative direct proof of \cite[Corollary 5.8]{aHLRS15} using results from Section \ref{sec:toeplitz}. 

The next lemma follows from \cite[Theorem 5.3]{aHLRS15}, but we give a direct proof. 

\begin{lemma} \label{lem:KMSlimit}
Let $E$ be a finite graph and let $v\in E^0$. If $\beta_v \geq 0$, then there exists a KMS$_{\beta_v}$ state $\phi$ on $\T C^*(E)$ supported in $E^\infty$. 
\end{lemma}
\begin{proof}
First, we show $Z_v(\beta_v) = \infty$. Let $T=T(E,v)$, and let $A_n = \{ v \in T^0 \colon |v|=n \}$. Then, we have $Z_v(\beta) = \sum_{n=0}^\infty |A_n|e^{-n\beta}$ for $\beta>\beta_v$. Since $T$ is a periodic tree with $\br T = e^{\beta_v}$ by Proposition \ref{prop:betac=dim}, we have 
\[ \underset{\Pi}{\inf} \sum_{e \in \Pi} e^{-|r(e)|\beta_v} > 0 \]
by \cite[Theorem 3.8]{LyoPer}, where $\Pi \subseteq T^1$ runs through all cutsets of $T$. Applying this to $\Pi_n := r^{-1}(A_n)$, we obtain
\[ \underset{n \in \Nz}{\inf} |A_n|e^{-n\beta_v} >0, \]
which implies that $Z_v(\beta_v)=\infty$. 

Since $E$ is finite, $\T C^*(E)$ is unital. Therefore, $\KMS_{\beta}(\T C^*(E),\alpha)$ is compact by \cite[Theorem~5.3.30]{BR}, so that the net $(\phi_{v,\beta})_{\beta>\beta_v}$ has a convergent subnet, $(\phi_{v,\beta_i})_i$. Let $\phi :=\lim_i \phi_{v,\beta_i}$, and let $m$ be the $e^{\beta_v}$-conformal probability measure on $E^{\leq\infty}$ associated to $\phi$. Then $m_{v,\beta_i}\to m$ weak* (viewed as elements of the dual space of $C(E^{\leq\infty})$). 
Since $Z_v(\beta_v)=\infty$, we have
\[ m(\{\mu\}) = \lim_{i} m_{v,\beta_i}(\{\mu\}) = \lim_i \frac{e^{-\beta_i |\mu|}}{Z_v(\beta_i)} =0 \]
for any $\mu \in E^*v$, which implies that $m$ is concentrated in $E^\infty$. 
\end{proof}

We use a formula from linear algebra, which is a consequence of triangularisation of the adjacency matrix $A_E$ by permuting blocks. See \cite[Section 2.3 and  Section 4]{aHLRS15} for the proof. 
\begin{lemma} \label{lem:seneta}
Let $E$ be a finite graph. Then, we have
\[ \rho(A_E) = \max \{\rho(A_C) \mid C \in \pi(E) \}.  \]
\end{lemma}

For a square matrix $A$, $\rho(A)$ denotes the spectral radius of $A$. 
We establish a relation between $\rho(A_E)$ and critical inverse temperatures.
Note that a finite graph $E$ has at least one nontrivial cycle (that is, a  nontrivial path $\mu$ with $s(\mu)=r(\mu)$) if and only if there exists $v \in E^0$ such that $\beta_v \geq 0$. 
\begin{proposition} \label{prop:tree}
Let $E$ be a finite graph which contains at least one nontrivial cycle. 
We have
\begin{equation}
    \log \rho (A_E) = \max_{v \in E^0} \dim \partial T(E,v) = \max_{v \in E^0} \beta_v. 
\end{equation}
\end{proposition}
\begin{proof}
The second equality follows from Proposition \ref{prop:betac=dim}. 
For each $v\in E^0$ with $\beta_v \neq -\infty$, it follows from Lemma \ref{lem:KMSlimit} that there exists a $\beta_v$-harmonic vector, which implies that $e^{\beta_v}$ is an eigenvalue of $A_E$. 
Hence we have $\rho (A_E) \geq e^{\beta_v}$ for all $v \in E^0$.

Suppose that $\beta>0$ satisfies $\max_{v \in E^0} Z_v(\beta) < \infty$. Then, for any $\beta' \geq \beta$, all extremal KMS$_{\beta'}$ states are of type I by Theorem \ref{thm:graph}, which implies that there is no $\beta'$-harmonic vector. By the weak form of the Perron--Frobenius Theorem (see, for instance, \cite[Theorem~4.2]{Minc}), there exists a nonnegative eigenvector of $A_E$ with eigenvalue $\rho(A_E)$. After normalising, this vector becomes a $\log \rho(A_E)$-harmonic vector, so we must have $\log \rho(A_E) < \beta$.
Consequently, we have  
\[ \log \rho(A_E) \leq \tilde{\beta} := \inf \{\beta>0 \colon \max_{v \in E^0} Z_v(\beta) < \infty \}. \]
Then, there exists $v \in E^\infty$ such that $Z_v(\tilde{\beta}) = \infty$ (see the proof of Lemma \ref{lem:KMSlimit}). Hence, we have
\[ \log \rho(A_E) \leq \tilde{\beta} \leq \beta_v, \]
which completes the proof. 
\end{proof}

\begin{theorem}[{\cite[Corollary~5.8 and Lemma~7.2]{aHLRS15}}] \label{thm:aHLRSmain}
Let $E$ be a finite graph, and let $v \in E^0$ such that $\beta_v>0$. 
Then we have 
\[ \beta_v = \max \{\log \rho(A_C) \colon C\in \pi(E), CE^*v \neq \emptyset \}. \]
\end{theorem}
\begin{proof}
Let $S = s(E^*v)$. Since we have $T(E,v) \cong T(E_S,v)$, the critical inverse temperature of $v$ is the same in both $\T C^*(E)$ and $\T C^*(E_S)$. In addition, since $T(E_S, w)$ is isomorphic to a subtree of $T(E_S,v)$ for any $w \in S$, we have 
\[ \dim \partial T(E_S,w) \leq \dim \partial T(E_S,v) \]
for all $w \in S$. 
Hence, by Proposition~\ref{prop:tree}, \[ \log \rho (A_S) = \max_{w \in E^0_S} \dim \partial T(E,w) = \dim \partial T(E,v) =\beta_v. \] By Lemma \ref{lem:restriction} and Lemma \ref{lem:hereditary}, we have \[ \pi(E_S) = \{C \in \pi(E) \colon CE^*v \neq \emptyset \}, \] so the claim follows from Lemma \ref{lem:seneta}.
\end{proof}

\begin{remark}
\label{rmk:criticalinversetemp}
Let $E$ be a countable graph. If $v, w \in E^0$ satisfy $vE^*w \neq \emptyset$, then $T(E,v)$ can be embedded into $T(E,w)$, which implies 
\[ \beta_v = \log \ugr T(E,v) \leq \log \ugr T(E,w) = \beta_w \] as in the proof of Theorem \ref{thm:aHLRSmain}. 
In particular, if $v,w \in E^0$ 
belong to the same strongly connected component of $E$, then  $\beta_v=\beta_w$.
\end{remark}

\subsection{KMS states at critical inverse temperatures}
The critical inverse temperatures $\beta_v$ are associated to vertices, but as we have seen in Remark \ref{rmk:criticalinversetemp}, $\beta_v$ only depends on the equivalence class (that is, strongly connected component) of $v$. For this reason, until the end of this section, critical inverse temperatures are denoted by $\beta_C$ instead of $\beta_v$, where $C \in \pi(E)$. Namely, $\beta_C=\beta_v$ for any $v \in C$. 

For a finite graph $E$, the structure of the KMS states on $\T C^*(E)$ is completely parametrised by \cite[Theorem~5.3]{aHLRS15}.
In this subsection, we review their work with a slight reformulation in order to clarify how the strongly connected components are related to critical inverse temperatures and KMS states. 

The set $\pi(E)$ of strongly connected components has a canonical partial ordering defined by $C'\leq C$ if and only if $C'E^*C\neq\emptyset$. The following notion of minimality is introduced in \cite[Section~4]{aHLRS15}: $C\in\pi(E)$ is a minimal critical component if $\rho(A_C)=\rho(A_E)$ and $C$ is minimal in the set $\{C'\in\pi(E) : \rho(A_{C'})=\rho(A_E)\}$ with respect to the restriction of ``$\leq$''. For our purposes here, it will be helpful to consider a different notion of minimality, which we now define.

\begin{definition}
\label{def:min}
Let $E$ be a finite graph. 
A strongly connected component $C \in \pi(E)$ is \emph{minimal} if whenever $C' \in \pi(E)$ satisfies $C'E^*C \neq \emptyset$, then $C'=C$ or $\rho(A_{C'}) < \rho (A_C)$. The set of all minimal strongly connected components with $\beta_C >0$ is denoted by $\pmc(E)$. 
\end{definition}

As we will see in this section, for a finite graph $E$ such that there is a vertex $v\in E^0$ with $\beta_v>0$, $C \in \pi(E)$ is minimal if and only if $C$ is a minimal critical component in $E_{E^0\setminus H_{\beta}}$ for some $\beta>0$ in the sense of \cite{aHLRS15}, where 
\begin{equation}\label{eqn:Hbeta}
H_\beta = \overline{\bigcup \{C \in \pi(E) \colon \rho(A_C) > e^\beta\}}.
\end{equation}
is the hereditary set defined in \cite[Theorem~5.3]{aHLRS15}. The purpose of this section is to give a parametrisation of KMS states that does not use $H_\beta$.

We see in the next proposition that only minimal strongly connected components are relevant to critical inverse temperatures. 

\begin{proposition} \label{prop:partialorder}
Let $E$ be a finite graph. 
For any $C \in \pmc(E)$, we have $\beta_C=\rho(A_C)$. Conversely, for any $v \in E^0$ with $\beta_v>0$, there exists $C \in \pmc(E)$ such that $CE^*v \neq \emptyset$ and $\beta_v=\beta_C$. In particular, $\pmc(E) \neq \emptyset$ if and only if there is at least one vertex $v \in E^0$ such that $\beta_v>0$.
\end{proposition}
\begin{proof}
The first claim follows from the definition of minimal components and Theorem \ref{thm:aHLRSmain}. We prove the second claim. Define a partial order $\leq$ on $\cC := \{C \in \pi(E) \colon \rho(A_C) = e^{\beta_v},\ CE^*v\neq\emptyset \}$ by $C\leq C' \Longleftrightarrow CE^*C' \neq \emptyset$. Since $\cC$ is a finite set and is nonempty by Theorem \ref{thm:aHLRSmain}, there exists a minimal element $C\in \cC$. 
If $C' \in \pi(E)$ satisfies $C'E^*C \neq \emptyset$, then $C'E^*v \neq \emptyset$, which implies $\rho(A_{C'}) \leq e^{\beta_v}$ by Theorem \ref{thm:aHLRSmain}. If $\rho(A_{C'}) = e^{\beta_v}$, then $C'\in\cC$, so that minimality of $C$ in $\cC$ implies $C'=C$. Hence, $C \in \pmc(E)$. 
\end{proof}

We are interested here in the KMS states from \cite[Theorem 4.3]{aHLRS15}. We begin with an observation about properties of these KMS states. 

\begin{proposition}[{\cite[Theorem~4.3]{aHLRS15}}] \label{prop:existence}
Let $E$ be a finite graph with $\pmc(E) \neq \emptyset$, and let $C \in \pmc(E)$. 

Then, there exists a KMS$_{\beta_C}$ state $\psi_C$ of $C^*(E)$ supported in $E^\infty$ satisfying 
\begin{equation}
\begin{cases}
\psi_C(p_w) >0 & \mbox{ if } w \in C, \mbox{and} \\
\psi_C(p_w) =0 & \mbox{ if } w \not\in s(E^*C).  
\end{cases}
\end{equation}
\end{proposition}
\begin{proof}
Let $S=s(E^*C)$. Since $C \in \pi(E_S)$ is a minimal critical component in the sense of \cite[Section~4]{aHLRS15}, we apply \cite[Theorem~4.3]{aHLRS15} to $E_S$ to obtain a KMS$_\beta$-state $\tilde{\psi}_C$ of $C^*(E_S)$. From the proof of  \cite[Theorem~4.3]{aHLRS15}, $\tilde{\psi}_C$ comes from a normalised $\beta$-harmonic vector, so that $\tilde{\psi}_C$ is supported in $E_S^\infty$. Let $\tilde{\harm}$ be the normalised $\beta_C$-harmonic vector of $E_S$ associated to $\tilde{\psi}$. Since $E^0\setminus S$ is hereditary by Lemma \ref{lem:hereditary}, $A_E$ is of the form 
\[ A_E = \begin{bmatrix} A_{E^0\setminus S} & 0 \\ * & A_S\end{bmatrix}. \]
Hence, 
\[ \harm = \begin{bmatrix} 0 \\ \tilde{\harm} \end{bmatrix} \]
is a normalised $\beta_C$-harmonic vector of $E$. Let $\psi_C$ be the KMS$_{\beta_C}$ state associated to $\harm$, and we show that $\psi_C$ satisfies the desired property. By definition, $\psi_C(p_w) =0$ if $w \not\in s(E^*C)$. From the construction of $\tilde{\psi}_C$, $[\harm_w]_{w \in C}$ is the Perron--Frobenius eigenvector of $A_C$ (see \cite[Theorem~4.1]{Minc}). Hence, $\harm_w > 0$ for $w \in C$, so that $\psi_C(p_w)=\chi_w>0$ for every $w \in C$. 
\end{proof}

\begin{remark}\label{rmk:coincidence}
Let $C \in \pmc(E)$, let $H$ be a hereditary subset such that $E^0\setminus H$ contains $s(E^*C)$, and suppose that $C$ still attains the maximum of 
\[ \rho(A_{E^0\setminus H}) = \max \{ \rho(A_{C'}) \colon C' \in \pi(E^0\setminus H)\} \]
from Lemma \ref{lem:seneta}. 
We apply \cite[Theorem 4.3]{aHLRS15} to $E_{E^0\setminus H}$ and $C$ to obtain a KMS$_\beta$-state $\psi'$ of $C^*(E)$ by the same argument as in Proposition \ref{prop:existence}. Let $\harm'$ is the $\beta_C$ harmonic vector associated to $\psi'$. 
Then we have 
\begin{equation}
\harm_w' = \sum_{n=1}^\infty \sum_{w' \in C} |wE^nw'| e^{-n\beta_C}\harm_{w'}'
\end{equation}
for any $w \in E^0\setminus C$ from the construction of \cite[Theorem 4.3]{aHLRS15} (observe that if $wE^*C = \emptyset$, then $\harm_w'=0$ anyway). 
Since the $\beta_C$-harmonic vector $\harm$ associated to $\psi_C$ satisfies the same relation and $\harm'_w = \harm_w$ for any $w \in C$, we have $\psi'=\psi_C$. In particular, the KMS state in Proposition \ref{prop:existence} is the same state as in \cite[Theorem 5.3]{aHLRS15}, which follows from the above argument applied to $H_{\beta_C}$, where $H_{\beta_C}$ is the hereditary set defined in Equation \eqref{eqn:Hbeta}.
\end{remark}

Now we establish a description of KMS states of $C^*(E)$ supported in $E^\infty$, which is a modification of \cite[Theorem 5.3]{aHLRS15}. 

\begin{theorem}[{\cite[Theorem 5.3]{aHLRS15}}] \label{thm:aHLRSmain2}
Let $E$ be a finite graph, and let $\beta>0$. If $\varphi$ is a KMS$_\beta$ state of $C^*(E)$ supported in $E^\infty$, then we have
\[ \varphi \in \mathrm{co} \{ \psi_C \colon C \in \pmc(E),\ \beta_C=\beta\}, \]
where $\mathrm{co}$ denotes the convex hull. Moreover, the KMS$_{\beta_C}$ state $\psi_C$ is extremal for any $C \in \pmc(E)$. 
\end{theorem}

\begin{proof}
In order to apply \cite[Theorem 5.3]{aHLRS15}, let $H_\beta$ be the hereditary set in Equation \eqref{eqn:Hbeta}. 
By \cite[Theorem 5.3 (a)]{aHLRS15}, we may assume that $E^0 \neq H_\beta$. Let $S_\beta := E^0\setminus H_\beta$. First, we show $\rho(A_{S_\beta}) \leq e^\beta$.
If $C \in \pi(E_{S_\beta})$ satisfies $\rho(A_C)>e^\beta$, then $C \subseteq H_\beta$ by definition of $H_\beta$, which is a contradiction. Hence we have $\rho(A_C) \leq e^\beta$ for all $C \in \pi(E_{S_\beta})$, which implies that 
\[ \rho(A_{S_\beta}) = \max \{ \rho(A_C) \colon C \in \pi(E_{S_\beta})\} \leq e^\beta \]
by Lemma \ref{lem:seneta}. 

If $\beta > \log \rho(A_{S_\beta})$, then there is no KMS$_\beta$ state supported in $E^\infty$ by \cite[Theorem 5.3(b)]{aHLRS15}. Therefore, we may assume that $\beta = \log \rho(A_{S_\beta})$. Then, $\mathrm{mc}(E\setminus H_\beta)$ in \cite{aHLRS15} coincides with 
\[ \{ C \in \pmc(E) \colon \beta_C=\beta\}. \]
Now the claim follows from Remark \ref{rmk:coincidence}. 
\end{proof}

\subsection{Factor types of KMS states}

We are now ready to prove the main theorem in this section. We determine the type of the KMS state $\psi_C$ in Proposition \ref{prop:existence}.

\begin{theorem}\label{thm:typedetermination}
Let $E$ be a finite graph such that $\pmc(E)\neq \emptyset$. 
Let $C \in \pmc(E)$, and let $\psi_C$ be the KMS state in Proposition \ref{prop:existence}. 
Then $\psi_C$ is of type III$_{e^{-s\beta_C}}$, where $s$ is the greatest common divisor of lengths of nontrivial cycles in $E_C$. 
\end{theorem}

\begin{proof}
Let $\lambda=e^{-s\beta_C}=\rho(A_C)^{-s}$, and let $\varphi$ be the composition of $\psi_C$ with the canonical quotient map $\T C^*(E) \to C^*(E)$. 
By \cite[Theorem~4.1]{FR}, $\T C^*(E_{\bar{C}})$ is naturally identified with a subalgebra of $\T C^*(E)$. Let $\psi$ be the normalisation of the restriction of $\varphi$ to $\T C^*(E_{\bar{C}})$. By Proposition~\ref{prop:existence}, we have $\varphi(p_v)>0$ for all $v \in C$, so that $\psi$ is nonzero. By the construction of $\varphi$, we have $\varphi(p_v)=0$ if $v \in E^0$ and $vE^*C = \emptyset$ (see Remark \ref{rmk:coincidence}). In particular, we have $\psi(p_v)=0$ for all $v \in \bar{C}\setminus C$, which implies that the $e^{\beta_C}$-conformal probability measure $m$ is concentrated in $\cZ'(C)$.

We claim that it suffices to show that $\psi$ is a factor state and is of type III$_\lambda$. Suppose that $\psi$ satisfies this claim. 
 
We have 
$p_C \T C^*(E) p_C \cong p_C \T C^*(E_{\bar{C}}) p_C$, because for any $\mu,\nu \in E^*$, we have $p_C s_{\mu} s_{\nu}^* p_C \neq 0 $ only if $s(\mu), s(\nu) \in C$. Under the identification $p_C \T C^*(E) p_C \cong p_C \T C^*(E_{\bar{C}}) p_C$, the restriction of $\varphi$ to $p_C \T C^*(E) p_C$ is equal to the restriction of $\psi$ to $p_C \T C^*(E_{\bar{C}}) p_C$ up to a positive constant (note that both restrictions are nonzero because $\varphi(p_C)>0$). We can see that the type of $\varphi$ is the same as the type of $\psi$, since taking the corner by a nonzero projection does not change the $S$-invariant (see \cite[Corollaire 3.2.8]{Con}). Hence, $\varphi$ is of type III$_\lambda$, this completes the proof of our claim. 

We now show that $\psi$ is a factor state and is of type III$_\lambda$. Let $K_C := E_C^\infty \subseteq E^\infty$, and we show 
\begin{equation} \label{eqn:decompose}
\cZ'(C) \setminus K_C \subseteq \bigcup \{ f \cZ'(w) \colon w \in \bar{C} \setminus C,\ f \in CE^*w \}. 
\end{equation}
Let $x = x_1x_2 \cdots \in \cZ'(C) \setminus K_C$. Then there exists $n$ such that $s(x_n) \in C$ and $r(x_n) \in \bar{C} \setminus C$. 
Then, $x \in f \cZ'(w)$ for $f = x_1 \cdots x_n$ and $w=r(x_n)$. Hence the inclusion \eqref{eqn:decompose} holds. 
Since the right hand side of \eqref{eqn:decompose} is a countable union of $m$-null set, $\cZ'(C) \setminus K_C$ is a $m$-null set. 
We can see that $K_C$ is a closed invariant set of $\cG_{E_{\bar{C}}}^{(0)}$ from the property that if $v \in \bar{C}$ and $vE^*C \neq \emptyset$, then $v \in C$. 
Since the restriction groupoid of $\cG_{E_{\bar{C}}}$ to $K_C$ is isomorphic to $\cG_{E_C}$, we see that $\psi$ factors through 
\[ \T C^*(E_{\bar{C}}) \to C^*(E_{\bar{C}}) \cong C^*(\cG_{E_{\bar{C}}}) \to C^*(\cG_{E_C}) \cong C^*(E_C),  \]
where the third homomorphism is the canonical quotient map associated with the invariant set $K_C$ (cf.~\cite[Proposition~10.3.2]{SSW}). 
By \cite[Theorem 3.1(2)(b)]{LLNSW}, the KMS$_\beta$ state of $C^*(E_C)$ uniquely exists, which implies that $\psi$ is a factor state. Moreover, $\psi$ is of type III$_\lambda$ by \cite[Theorem 3.1(2)(b)]{LLNSW} (see also \cite[Theorem 5]{Th19}). 
\end{proof}

Finally, note that the number $s$ in Theorem \ref{thm:typedetermination} is equal to the period of the adjacency matrix $A_C$ of the graph $E_C$ (see \cite[Chapter~1,~Definition~1.6]{Sen}). In \cite{EFW}, which can be seen as the first result on the factor types of KMS states of graph algebras, the factor type is described in terms of the period of the adjacency matrix. 

\section{Examples}
\label{sec:examples}

In this section, we apply our general results to recover graph-theoretic invariants as explicit C*-dynamical invariants of $(\T C^*(E),\alpha)$ for several concrete example classes, including finite graphs in Section~\ref{sec:finite}, Bratteli diagrams in Section~\ref{sec:BD}, and opposite graphs of Bratteli diagrams Section~\ref{sec:BDop}. In addition, we give concrete examples of type computations of KMS states of finite graph algebras in Section~\ref{sec:finite}.

\subsection{Finite graphs}
\label{sec:finite}

Let $E$ be a finite directed graph with adjacency matrix $A_E$. Consider the generating function for the numbers $|E^n|$ of finite paths of length $n$ in $E$ (see \cite[Chapter~1.8]{CDS}):
\[
H_E(t)=\sum_{n=0}^\infty|E^n|t^n.
\]
Our results imply that this generating function has the following natural interpretation in terms of the C*-dynamical system $(\T C^*(E),\alpha)$:

\begin{proposition}
\label{prop:genfnc}
	Let $E$ be a finite directed graph, and denote by $\alpha$ the gauge action on $\T C^*(E)$. For every $\beta>\log\rho(A_E)$, we have
	\[
	H_E(e^{-\beta})=\sum_{\phi\in \E(\KMS_{\beta}(\T C^*(E), \alpha))}Z_\phi(\beta).
	\]
	In particular, the function $H_E(t)$ is explicitly given as a C*-dynamical invariant of $(\T C^*(E),\alpha)$.
\end{proposition}

\begin{proof}
First, \cite[Theorem~3.1]{aHLRS13} combined with the type computation in Theorem~\ref{thm:Toeplitz} implies that for each $\beta>\log\rho(A_E)$, the set $\E(\KMS_{\beta}(\T C^*(E), \alpha))$ coincides with $\{\phi_{v,\beta} : v\in E^0\}$. By Theorem~\ref{thm:Toeplitz}(ii), $Z_{\phi_{v,\beta}}(\beta)=Z_v(\beta)$. It remains to observe that 
	\[
	H_E(e^{-\beta})=\sum_{v\in E^0}\sum_{n=0}^\infty|E^nv|e^{-\beta n}=\sum_{v\in E^0}Z_v(\beta). 
	\qedhere	
	\]
\end{proof}

\begin{remark}
\label{rmk:lowtemppar}
In the case of finite graphs, the low-temperature KMS states of $(\T C^*(E),\alpha)$ are parameterised by \cite[Thoerem~3.1]{aHLRS13}. Proposition~\ref{prop:tree} combined with Theorem~\ref{thm:KMSwithgrowthcond} yield the parameterisation given in \cite[Thoerem~3.1]{aHLRS13} in the case where $E$ contains at least one nontrivial cycle.
\end{remark}

As an immediate consequence, we have: 

\begin{corollary}
	Let $E$ and $F$ be finite directed graphs. If there exists an isomorphism of C*-dynamical systems $(\T C^*(E),\alpha^E)\cong (\T C^*(F),\alpha^F)$, then $H_E(t)=H_F(t)$.
\end{corollary}

\begin{remark}
That $H_E(t)$ is given as an explicit C*-dynamical invariant of $(\T C^*(E),\alpha^E)$ is analogous to the fact the Dedekind Zeta function of a number field is given explicitly as a C*-dynamical invariant of the C*-dynamical system associated with the full $ax+b$-semigroup over a ring of algebraic integers in a number field (cf. \cite[Corollary~7.7]{BLT}). In the setting of Toeplitz algebras of graphs, the so-called fixed target partition functions $Z_v(s)$, $v\in E^0$, play a role similar to that played by the partial Dedekind Zeta functions in the number-theoretic setting in \cite{BLT}.
\end{remark}

We shall demonstrate a calculation of partition functions and factor types at the critical inverse temperatures in a concrete example. Note that the next example is also examined in \cite[Example~6.2]{aHLRS15} (not \cite[Example~6.1]{aHLRS15}, since the convention in \cite{aHLRS15} is opposite to ours).

\begin{example}
\label{ex:finite1}
We consider the following graph $E$: $E^0=\{v,w\}, E^1=\{e_1,e_2,e_3,f,g_1,g_2\}$, where $s(e_i)=r(e_i)=w$ for $i=1,2,3$, $s(g_j)=r(g_j)=v$ for $j=1,2$, $s(f)=v$ and $r(f)=w$. 

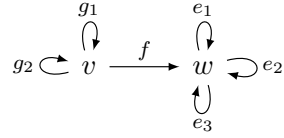
\begin{equation}
\begin{tikzcd}[
arrow style=tikz,
>={latex} 
]
v\arrow[loop above,"g_1"]\arrow[loop left,"g_2"]\arrow[r,"f"] & w\arrow[loop above,"e_1"]\arrow[loop right,"e_2"]\arrow[loop below,"e_3"] &
\end{tikzcd}
\end{equation}

Then, $\beta_v = \log 2$ since $T(E,v)$ is a binary tree (see Proposition \ref{prop:betac=dim}), and $\beta_w = \log 3$ by Theorem \ref{thm:aHLRSmain}. However, the critical inverse temperatures can also be calculated by determining partition functions explicitly. The adjacency matrix of $E$ is given by
\[
A_E = \begin{bmatrix} 2 & 1 \\ 0 & 3 \end{bmatrix}.  
\]
By diagonalising $A_E$, we obtain 
\[ A_E^n = \begin{bmatrix} 2^n & 3^n-2^n \\ 0 & 3^n \end{bmatrix}\] 
for each $n \in \Nz$. Hence, the partition functions are 
\begin{align}
Z_w(\beta) &= \sum_{n=0}^\infty (3^n + (3^n-2^n))e^{-n\beta} 
= \frac{2}{1-3e^{-\beta}} - \frac{1}{1-2e^{-\beta}}, \\
Z_v(\beta) &= \sum_{n=0}^\infty 2^ne^{-n\beta} 
= \frac{1}{1-2e^{-\beta}}. 
\end{align}
Now it is clear that the abscissas of convergence of $Z_v$ and $Z_w$ are given by $\beta_v= \log 2$ and $\beta_w=\log 3$. 
The list of partition functions $\{Z_v, Z_w\}$ is a C*-dynamical invariant. In addition, the generating function for finite paths  
\[ H_E(t) = Z_v(-\log t) + Z_w(-\log t) = \frac{2}{1-3t} \]
is a C*-dynamical invariant. 

We apply Theorem \ref{thm:typedetermination} to determine the type of KMS states on $\T C^*(E)$ factoring through $C^*(E)$. 
We have $\pi(E)=\{C_v,C_w\}$, where $C_v=\{v\}$ and $C_w=\{w\}$. In this case, we have $\pi(E)=\pmc(E)$. Hence, there is a KMS$_{\beta_v}$ state $\psi_{C_v}$ and a KMS$_{\beta_w}$ state $\psi_{C_w}$ on $C^*(E)$, and there are no other KMS states on $\T C^*(E)$ factoring through $C^*(E)$. Since both $E_{C_v}$ and $E_{C_w}$ have cycles of length $1$, $\psi_{C_v}$ is of type III$_{1/2}$ and $\psi_{C_w}$ is of type III$_{1/3}$.
\end{example}

Next, we visit a more complicated example. Although the calculation of partition functions is harder for such examples, the type computation is possible by hand. 
We use the following well-known lemma (for instance, \cite[Example 1.1]{LyoPer}) in the next example:  

\begin{lemma} \label{lem:geometricmean}
Let $n \in \Nz$, and let $k_0,\cdots, k_{n-1}$ be a finite sequence of positive integers. 
Let $T$ be a rooted tree such that each 
vertex in the $m$-th level has $k_i$ children if $m \equiv i \mod n$. 
Then, 
\[ \br T = \sqrt[n]{k_0 \cdots k_{n-1}}. \]
\end{lemma}

For the definition of branching numbers, see \cite[p.74]{LyoPer}. However, we only use the fact $\br T = \ugr T$ for a periodic tree $T$ by \cite[Theorem~3.8]{LyoPer}. 

\begin{proof}
For each $l \in \Nz$, the number of vertices $a_{ln}$ in the $ln$-th level is equal to $(k_0 \cdots k_{n-1})^l$. 
Since, $T$ is a periodic tree, we have
\[ \br T = \lim_{l \to \infty} a_{ln}^{1/ln} = \sqrt[n]{k_0 \cdots k_{n-1}}. \]
\end{proof}

\begin{example}
\label{ex:finite2}
Let $E$ be the graph depicted as follows: 
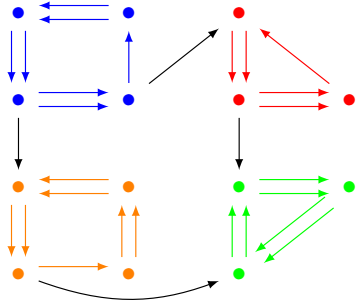
\begin{equation}
	\begin{tikzcd}[
		arrow style=tikz,>={latex}]
	 \bullb \arrow[d, blue, shift left] \arrow[d, blue, shift right] & \bullb  \arrow[l, blue, shift left] \arrow[l, blue, shift right] & \bullr \arrow[d, red, shift left] \arrow[d, red, shift right]&  \\
	  \bullb \arrow[d]\arrow[r, blue, shift left]  \arrow[r, blue, shift right]& \bullb  \arrow[u, blue]\arrow[ru]&\bullr\arrow[r, red, shift left] \arrow[r, red, shift right]\arrow[d]& \bullr\arrow[ul,red]\\  
	 \bullo \arrow[d, orange, shift left] \arrow[d, orange, shift right]& \bullo  \arrow[l, orange, shift left] \arrow[l, orange, shift right]& \bullg \arrow[r, green, shift left] \arrow[r, green, shift right]  &\bullg \arrow[dl, green, shift left] \arrow[dl, green, shift right]\\
	 \bullo \arrow[r, orange, shift left] \arrow[rr,bend right=20]& \bullo  \arrow[u, orange, shift left] \arrow[u, orange, shift right] & \bullg\arrow[u, green, shift left] \arrow[u, green, shift right]&
	\end{tikzcd}
\end{equation}

Here, the colour is only used to mark different strongly connected components of $E$: We have $\pi(E) = \{C_i \colon i=1,2,3,4\}$, where $E_{C_1}$ is the green part, $E_{C_2}$ is the orange part, $E_{C_3}$ is the blue part, and $E_{C_4}$ is the red part. Fix $v_i \in C_i$, and let $T_i=T(E_{C_i}, v_i)$. In order to determine minimal components, we calculate $\rho(A_{C_i})$. Needless to say, it is easy to calculate $\rho(A_{C_i})$ from the characteristic polynomials, since $E$ is a finite graph. 
However, there is a intuitive way to calculate it from the formula $\rho(A_{C_i}) = \br T_i$ from Proposition~\ref{prop:tree}. 
First, $T_1$ is a binary tree, so $\br T_1 = 2$. Second, $T_4$ is the tree in Lemma \ref{lem:geometricmean} with $n=3$, $k_1=1$ and $k_2=k_3=2$, so that $\br T_4 = \sqrt[3]{4}$. Finally, $T_2$ and $T_3$ are isomorphic, and are the tree in Lemma \ref{lem:geometricmean} with $n=4$, $k_1=1$ and $k_2=k_3=k_4=2$, which implies that $\br T_3 = \br T_4 = \sqrt[4]{8}$. Since $\br T_1 > \br T_2 = \br T_3 > \br T_4$, we conclude that $\pmc(E)=\{C_1,C_3\}$. 

Hence, $\psi_{C_1}$ is a KMS$_{\beta_1}$ state on $C^*(E)$ for $\beta_1=\log \br T_1 = \log 2$, and $\psi_{C_3}$ is a KMS$_{\beta_3}$ state on $C^*(E)$ for $\beta_3=\log \br T_3=\frac{3}{4}\log 2$. Moreover, there are no other KMS states on $C^*(E)$. 
Let $s_i$ be the greatest common divisor of lengths of nontrivial cycles in $E_{C_i}$. Then $s_1=3$ and $s_3=4$. Hence, $\psi_{C_i}$ is of type III$_{\lambda_i}$, where
\[ \lambda_1 = e^{-\beta_1 s_1} = \frac18,\ \lambda_3 = e^{-\beta_3 s_3}= \frac18 \]
by Theorem \ref{thm:typedetermination}. In conclusion, both $\psi_{C_1}$ and $\psi_{C_3}$ are of type III$_{1/8}$. 
\end{example}

\subsection{Bratteli Diagrams}
\label{sec:BD}

Bratteli used a special kind of directed graph in his study of AF-algebras in \cite{Bra72} which are now referred to as \emph{Bratteli diagrams}. 

\begin{definition}
\label{def:BD}
	The countable directed graph $E$ is a \emph{Bratteli diagram} if there is a disjoint union decomposition $E^0=\coprod_{n=0}^\infty V_n$ such that
	\begin{enumerate}[\upshape(i)]
		\item $|V_0|=1$ and $r(s^{-1}(V_n))= V_{n+1}$ for every $n\geq 0$;
		\item $0<|V_n|<\infty$ for every $n\geq 0$; 
		\item $|r^{-1}\{v\}|<\infty$ for every $v\in E^0\setminus V_0$;
		\item $0<|s^{-1}\{v\}|<\infty$ for every $v\in E^0$.
	\end{enumerate}
\end{definition}

For a graph satisfying condition (i) from Definition~\ref{def:BD}, we can strengthen the conclusion of Theorem~\ref{thm:general}:

\begin{proposition}
	\label{prop:BD}
	Let $E$ and $E'$ be graphs satisfying condition (i) from Definition~\ref{def:BD}, and let $\alpha$ and $\alpha'$ be the gauge actions on $\T C^*(E)$ and $\T C^*(E')$, respectively. If there is an isomorphism of C*-dynamical systems $(\T C^*(E),\alpha)\cong (\T C^*(E'),\alpha')$, then there exists a bijection $E^0\to E'^0$, $v\mapsto v'$, that maps $V_n$ onto $V_n'$ for every $n\geq 0$ such that $|E^kv|=|E'^kv'|$ for every $v\in V_n$ and $k\geq 0$.
\end{proposition}
\begin{proof} 
By Theorem~\ref{thm:general}, there exists a bijection $E^0\to E'^0$ $,v\mapsto v'$ such that $|E^kv|=|E'^kv'|$ for all $k\geq 0$.
It remains to show that this map sends $V_n$ onto $V'_n$ for every $n\geq 0$. If $v\in V_n$, then there are no paths of length strictly greater than $n$ that terminate at $v$. Moreover, by assumption, for each $k\leq n$, there is at least one path of length $k$ from $v_0$ to $v$. Hence, $n=\max\{k : |E^kv|\neq 0 \}$ (note that $|E^kv|\in\Nz\cup\{\infty\}$).
From this, we see that $v\mapsto v'$ defines an inclusion from $V_n$ into $V'_n$ for every $n\geq 0$. Moreover, the inverse image of $V_n'$ under map $v\mapsto v'$ is contained in $V_n$. Hence, $v\mapsto v'$ defines a bijection from $V_n$ onto $V_n'$.
\end{proof}

\begin{example}
The conclusion of Proposition~\ref{prop:BD} does not imply that $E$ and $E'$ are isomorphic as graphs. For instance, consider the (non-isomorphic) Bratteli diagrams $E$ and $E'$ depicted by:
	
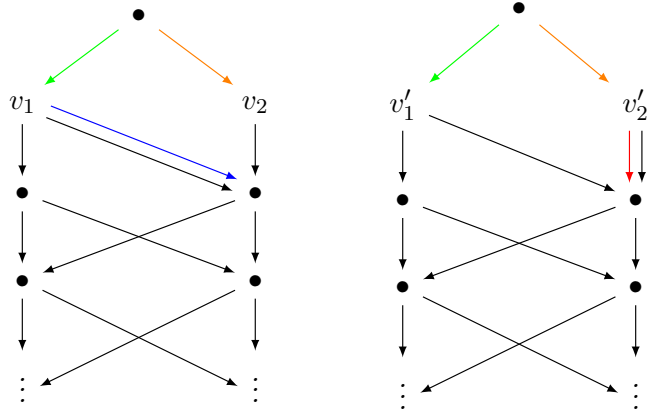
\begin{figure}[H]
		\centering
		\begin{tikzcd}[
		arrow style=tikz,
		>={latex} 
		]		
		& \bullet \arrow[dr,orange]\arrow[dl,green]  &\\
		v_1 \arrow[d] \arrow[drr,shift left=0.5ex,blue]\arrow[drr,shift right=0.5ex]&  & v_2 \arrow[d]\\
		\bullet\arrow[d]\arrow{drr} &  & \bullet \arrow[d]\arrow{dll}\\
		\bullet\arrow[d]\arrow{drr}  &  & \bullet \arrow[d]\arrow{dll} \\
		\vdots & &\vdots
		\end{tikzcd}
		\hspace{1cm}
		\begin{tikzcd}[
		arrow style=tikz,
		>={latex} 
		]		
		& \bullet \arrow[dr,orange]\arrow[dl,green]  &\\
		v_1' \arrow[d] \arrow[drr]&  & v_2' \arrow[d,shift left=0.5ex]\arrow[d,shift right=0.5ex,red]\\
		\bullet\arrow[d]\arrow{drr}  &  & \bullet \arrow[d]\arrow{dll} \\
		\bullet\arrow[d]\arrow{drr}  &  & \bullet \arrow[d]\arrow{dll} \\
		\vdots & &\vdots
		\end{tikzcd}
	\end{figure}
We will give a sketch of the proof that $(\T C^*(E),\alpha)\cong (\T C^*(E'),\alpha')$. The idea is similar to that used in \cite[Example~2.1]{BLRS}. There is a canonical (level-preserving) bijection $E^0\to E'^0$, $v\mapsto v'$. Let $a$ and $b$ denote the green and orange edges in $E$, respectively, and let $a'$ and $b'$ denote the green and orange edges in $E'$, respectively. Let $g$ denote the blue edge in $E$, and let $g'$ denote the red edge in $E'$. Let $E^1\to E'^1$, $e\mapsto e'$, be the bijection that sends $a$ to $a'$, $b$ to $b'$, and $g$ to $g'$ and is the canonical identification on all other edges. Let $v_1,v_2$ and $v_1',v_2'$ be the vertices at the first level of $E$ and $E'$, respectively, as indicated in the diagram. For $e\in E^1$, let 
\[
t_e:=\begin{cases}
 s_{a'}+s_{b'}s_{g'}s_{g'}^*& \text{ if } e=a,\\
s_{b'}-s_{b'}s_{g'}s_{g'}^* & \text{ if } e=b,\\
s_{e'} & \text { otherwise},
\end{cases}
\]
and for $v\in E^0$, put 
\[
q_v:=\begin{cases}
 p_{v_1'}+s_{g'}s_{g'}^* & \text{ if } v=v_1,\\
 p_{v_2'}-s_{g'}s_{g'}^*& \text{ if } v=v_2,\\
 p_{v'} & \text{ otherwise.}
\end{cases}
\]

A tedious but routine calculation shows that partial isometries $t_e$, $e\in E^1$, and the projections $q_v$, $v\in E^0$ in $\T C^*(E')$, satisfy the relations defining $\T C^*(E)$, so there exists a *-homomorphism $\T C^*(E)\to \T C^*(E')$ determined by $p_v\mapsto q_v$ for $v\in E^0$ and $s_e\mapsto t_e$ for $e\in E^1$. It is routine to write down the inverse for this map and to see that it is an $\Rz$-equivariant isomorphism.
\end{example}

\subsection{Opposite graphs of Bratteli diagrams}
\label{sec:BDop}

In general, there may be a stark contrast between the amount of information about $E$ contained in the C*-algebra $\T C^*(E)$ and amount of information about $E$ contained in the C*-dynamical system $(\T C^*(E),\alpha^E)$. In this section, we demonstrate this by considering opposite graphs of Bratteli diagrams.

\begin{proposition}
	\label{prop:oppBD}
If $E$ is the opposite graph of a Bratteli diagram, then  $\T C^*(E)\cong \bigoplus_{v\in E^0} \cK(\ell^2E^*v)$.
\end{proposition}
\begin{proof}
Since $E^\infty=\emptyset$, we have $\G_E^{(0)}=E^*$, which is a countable discrete space. Since every point in $E^*$ has trivial isotropy, and the orbits for the action of $\G_E$ on $E^*$ are precisely the sets $E^*v$ for $v\in E^0$, there is a disjoint union decomposition $\G_E=\bigsqcup_{v\in E^0}\G_E\vert_{E^*v}$, and thus an isomorphism $\T C^*(E)\cong C^*(\G_E)\cong \bigoplus_{v\in E^0} \cK(\ell^2E^*v)$. Here, we used that the restriction groupoid $\G_E\vert_{E^*v}$ is transitive (see \cite[Theorem~3.1]{MRW}).
\end{proof} 

\begin{remark}
	The isomorphism from Proposition~\ref{prop:oppBD} carries $\alpha^E$ to the time evolution $t\mapsto \bigoplus_v\Ad e^{itH_v}$, where $\Ad e^{itH_v}$ is the automorphism of $\cK(\ell^2E^*v)$ given by conjugation with $e^{itH_v}$. From this, one can deduce that $\E_{\rm I}(\KMS_{\beta}(\T C^*(E)))$ can be identified, as a topological space, with the countable discrete space $E_{\beta\textup{-reg}}^0$.
\end{remark}

As an immediate consequence of Proposition~\ref{prop:oppBD}, we obtain:
\begin{corollary}
For opposite graphs of Bratteli diagrams, the isomorphism class of the C*-algebra $\T C^*(E)$ does not depend on $E$.
\end{corollary}
Whereas, by Theorem~\ref{thm:general}, we know that if we keep track of the time evolution $\alpha^E$, then we can recover lots of combinatorial information about $E$. This is particularly stark in the case of opposite graphs of certain trees, which we now explain.

\begin{example}
For $d\geq 2$, the $d$-ary tree is the graph $T$ such that $T^0=\bigsqcup_{n=0}^\infty V_n$ for non-empty, finite sets $V_n$ such that $r(s^{-1}(V_n))\subseteq V_{n+1}$ for all $n$, $|V_0|=\{o\}$, every vertex emits exactly $d$ edges, and every vertex other than the root, $o$, receives exactly one edge.

Consider the C*-dynamical system $(\T C^*(T^{\rm op}),\alpha^{T^{\rm op}})$ of the opposite graph $T^{\rm op}$ of the $d$-ary tree $T$. Using Proposition~\ref{prop:betac=dim}, one can show that $\beta_v=\log d$ for every vertex $v$ of $T^{\rm op}$. Therefore, if $T$ and $T'$ are the $d$-ary and $d'$-ary trees, respectively ($d,d'\in\Zz_{>1}$), then we have $(\T C^*(T^{\rm op}),\alpha^{T^{\rm op}})\cong (\T C^*(T'^{\rm op}),\alpha^{T'^{\rm op}})$ if and only if $d=d'$. Whereas by Proposition~\ref{prop:oppBD}, we have $\T C^*(T^{\rm op})\cong \T C^*(T'^{\rm op})$ for all $d,d'\geq 1$.
\end{example}

\begin{remark}
The analogue of Proposition~\ref{prop:BD} for opposite graphs of Bratteli diagrams fails spectacularly. Indeed, if $T$ is the $d$-ary tree ($d\in\Zz_{>1}$), then any bijection $(T^{\rm op})^0\to (T^{\rm op})^0$ arises from an $\Rz$-equivariant isomorphism of the C*-dynamical system $(\T C^*(T^{\rm op}),\alpha^{T^{\rm op}})$; this can be seen as follows: By Proposition~\ref{prop:oppBD}, $\T C^*(T^{\rm op})\cong \bigoplus_{v\in (T^{\rm op})^0}\cK(\ell^2((T^{\rm op})^*v))$, and the bijection on the vertex set defines an $\Rz$-equivariant automorphism by permuting the summands in this decomposition.
\end{remark}

It is relatively easy to control the growth of lengths of paths in opposite graphs of Bratteli diagrams, so they can be combined to construct examples of C*-dynamical systems that exhibit wild phase transitions, as we now demonstrate.

\begin{proposition} \label{prop:wildphasetransition}
    For any countable infinite subset $C \subset (0,\infty)$, there exists a countable graph $E$ such that the set of critical inverse temperatures 
    \[ \{ \beta_v \colon v \in \bigcup_{\beta>0} E^0_{\beta\textup{-reg}} \}\]
    is equal to $C$.  
\end{proposition}

\begin{proof}
    Let $D = \exp C \subset (1,\infty)$, and fix $d \in D$. Let $\{a_n\}_{n=1}^\infty$ be a sequence of positive integers satisfying  
    \[ 0< \underset{n>0}{\inf} \frac{a_1 \cdots a_n}{d^n} \leq \sup_{n>0} \frac{a_1 \cdots a_n}{d^n} < \infty. \]
    For example, if we define $a_1$ to be the integer part of $d$, and $a_n$ to be the integer part of $d^n/(a_1\cdots a_{n-1})$ for $n>1$, then the sequence $\{a_n\}$ satisfies this property. Let $E_d$ be the graph defined by letting $E_d^0=\{v_k^{(d)} \colon k \in \Nz\}$ and putting $a_k$ edges from $v_{k}^{(d)}$ to $v_{k-1}^{(d)}$. Let $k \in \Nz$, and let $T_k = T(E_d,v_k^{(d)})$. Then, 
    \begin{align}
        \log \ugr T_k = \limsup_{n\to \infty} \frac{\log(a_{k+1}\cdots a_{k+n})}{n}=\log d.
    \end{align}
    Let $E$ be the graph obtained by connecting all $E_d$ to one vertex $o \in E^0$. Namely, 
    \[ E^0 = \{o\} \cup \bigcup_{d \in D} E_d^0, \]
    and we put the same edges as $E_d$ between vertices in $E_d^0$, and put an edge from $v_0^{(d)}$ to $o$. Note that $o \in E^0$ is not $\beta$-summable because $C$ is infinite. Hence, we have
    \[ \bigcup_{\beta>0} E^0_{\beta\textup{-reg}} = \bigcup_{d \in D} E_d^0 \]
    and for each $v=v_k^{(d)} \in E_d^0$ with $d \in D, k \in \Nz$, we have $\beta_v = \log d$ by Proposition \ref{prop:betac=dim}, which completes the proof. 
\end{proof}

\section{$\beta$-summability of infinite paths when $\beta<0$}
\label{sec:negbeta}
We have seen in Lemma~\ref{lem:xnotsummable} that for $\beta>0$, the orbit of an infinite path is never $\beta$-summable. On the other hand, if $\beta<0$, the orbit of a finite path is rarely $\beta$-summable (see Remark~\ref{rmk:nec} below), and the orbit of an infinite path can become $\beta$-summable. In this section, we shall determine when the orbit of an infinite path is $\beta$-summable in the case $\beta<0$.

\subsection{Necessary conditions for $\beta$-summability of infinite paths}
Given an infinite path $x=x_0x_1...\in E^\infty$, we let $[x]:=\G_Ex=\{y\in E^\infty : \exists n,m\geq 0 \text{ with }x_{i+n}=y_{i+m}\;\forall i\geq 0\}$ be the orbit of $x$. Note that $[x]$ is precisely the tail equivalence class of $x$.

\begin{lemma}
\label{lem:nec0}
   Let $x=x_0x_1...\in E^\infty$. 
   \begin{enumerate}[\upshape(i)]
       \item The orbit $[x]$ is consistent (see Section~\ref{sec:betasummable}) if and only if $x$ is not eventually periodic.
       \item If there exists an infinite sequence $(\gamma_n)_n$ in $\G_E$ with $s(\gamma_n)=x$, $r(\gamma_n)\neq r(\gamma_m)$ for $n\neq m$, and $\inf_{n}c(\gamma_n)>-\infty$, then $[x]$ is not $\beta$-summable for all $\beta<0$.
       \end{enumerate}
\end{lemma}
\begin{proof}
 Proof of (i): The infinite path $x\in E^\infty$ has non-trivial isotropy if and only if it is eventually periodic, see, for instance, the proof of \cite[Proposition~2.3]{BCW}. Moreover, it is easy to see that $[x]$ is consistent if and only if $x$ has trivial isotropy.
 
 Proof of (ii): Let $\beta<0$ and suppose $[x]$ is consistent, so that $l_x\colon [x]\to (0,\infty)$ is well-defined. We have
 \[
 \sum_{z\in[x]}l_x(z)^\beta\geq \sum_{n=1}^\infty e^{-\beta c(\gamma_n)}=\infty,
 \]
 so that $[x]$ is not $\beta$-summable. 
\end{proof}

Recall that a vertex $v \in E^0$ is called a \emph{source (in $E$)} if there is no edge $e$ such that $r(e)=v$. 
\begin{lemma}
	\label{lem:nec}
	Let $x=x_0x_1...\in E^\infty$ and $\beta\in(-\infty,0)$. If $[x]$ is $\beta$-summable, then the set $E^*s(x_i)$ is finite for every $i\geq 0$. Moreover, we have $s(x_i)\neq s(x_j)$ whenever $i\neq j$, and there exists $\mu\in E^*$ such that $s(\mu)$ is a source and $r(\mu)=s(x)$.
\end{lemma}
\begin{proof}
Suppose there exists $i\geq 0$ such that $|E^*s(x_i)|=\infty$. Let $\mu_1,\mu_2,...$ be any listing of $E^*s(x_i)\setminus\{s(x_i)\}$. For each $n\geq 1$, consider the element $\gamma_n:=(\mu_nx_ix_{i+1}...,|\mu_n|-i,x)\in\G_E$. We have $s(\gamma_n)=x$, $r(\gamma_n)\neq r(\gamma_m)$ for $n\neq m$ (because the $\mu_j$'s are distinct), and $c(\gamma_n)\geq -i$ for all $n$. Now Lemma~\ref{lem:nec0}(ii) implies that $[x]$ is not $\beta$-summable. This settles the first claim.

It remains to show that the second two properties follow from finiteness $E^*s(x_i)$ for every $i\geq 0$. First, suppose there exists $i\neq j$ such that $s(x_i)=s(x_j)$. Without loss of generality, we may assume that $i<j$. Consider the cycle $\mu:=x_i...x_{j-1}$ at $s(x_i)$. For each $n\geq 1$, we have $\mu^n\in E^*s(x_i)$, which contradicts that $E^*s(x)$ is finite.
Second, suppose there is no finite path $\mu\in E^*$ such that $s(\mu)$ is a source and $r(\mu)=s(x)$. Then there exists edges $e_{-i}$ for every $i\geq 1$ such that $r(e_{-1})=s(x)$ and $r(e_{-(i+1)})=s(e_{-i})$ for every $i\geq 1$. But now we have $e_{-n}...e_{-1}\in E^*s(x)$ for every $n\geq 1$, which contradicts that $E^*s(x)$ is finite.
\end{proof}

\begin{remark}
\label{rmk:nec}
If $\beta<0$, then by the same argument as in Lemma~\ref{lem:nec}, we see that $v\in E^0$ is $\beta$-summable if and only if $E^*v$ is finite. If $E^*v$ is finite, then the GNS representation of $\varphi_{\beta,v}$ is finite-dimensional for all $\beta$. Thus, we will focus on $\beta$-summability for infinite paths.
\end{remark}

\subsection{Necessary and sufficient conditions for $\beta$-summability of infinite paths}
\label{ssec:N&S}

Let $\beta\in(-\infty,0)$, and suppose $x\in E^\infty$ is such that $[x]$ is $\beta$-summable. Then by Lemma~\ref{lem:nec}, $x$ is tail equivalent to an infinite path beginning at a source. Hence, when we consider the tail equivalence class of $x$, we can always assume that $x$ begins at a source. 

For an infinite path $x=x_0x_1\cdots\in E^\infty$, let $V_x = \{ v \in E^0 \mid vE^*s(x_n) \neq \emptyset \mbox{ for some }n\geq 0\}$ 
and  $W_x = \{ s(x_j) \mid j=0,1,\cdots \}$.

In this section, we assume that the graph $E$ satisfies the following conditions: 
\begin{enumerate}
	\item There is an infinite path $x_E=x_0x_1\cdots \in E^\infty$ such that $v_0:=s(x_0)$ is a source. Fix such an infinite path $x_E$ and let $v_i:=s(x_i)$ for $i\geq 0$. Note that $W_{x_E} = \{v_0, v_1, \cdots \}$. 
	\item For every $v\in E_S^0$, we have $|vE_S^1|<\infty$ (that is, $E_S$ is row-finite), where $S=V_{x_E}$. 
	\item For $S=V_{x_E}$, there are no nontrivial cycles in $E_S$. 
	\item $V_{x_E} \setminus W_{x_E}$ is a finite set. 
\end{enumerate}

In this setting, the orbit of any vertex in $V_{x_E}$ is always a finite set. 
Such finite orbits are $\beta$-summable for all $\beta \in \Rz^*$, which gives rise to KMS states with finite-dimensional GNS representations. 
In this section, we are interested in extremal KMS states on $\T C^*(E)$ with infinite-dimensional GNS representations. Such KMS states are necessarily supported in $E^\infty$. 

We are interested in the $\beta$-summability of $[x_E]$, and more generally, $e^\beta$-conformal measures whose support contains $[x_E]$. We point out that $x_E$ has trivial isotropy, so that $[x_E]$ is always consistent (cf. Section~\ref{sec:betasummable}).
By Lemma~\ref{lem:nec}, conditions (1) and (3) are necessary to have non-trivial KMS$_\beta$ states of $\cT C^*(E)$ for negative $\beta$. Conditions (2) and (4) are just technical conditions; they imply that there are only finitely many edges whose source or range does not belong to $W_{x_E}$. Therefore, $\{v_n,v_{n+1},\cdots\}$ is a hereditary set for large $n$. 

The next lemma allows us to localise certain problems of KMS states to $V_{x_E}$. 
\begin{lemma} \label{lem:harmonicvectorextension}
    Suppose that a subset $S$ of $E^0$ satisfies that if $e \in E^1$ and $r(e) \in S$, then $s(e) \in S$. Then, any normalised $\beta$-harmonic vector $\tilde{\harm}$ of $E_S$ extends to a normalised $\beta$-harmonic vector $\harm$ of $E$ by 
    \begin{equation}
        \harm_v = \begin{cases} \tilde{\harm}_v & \mbox{if } v \in S \\
        0 & \mbox{if } v \not\in S. \end{cases}
    \end{equation}
\end{lemma}
\begin{proof}
    Let $A_E = [a_{vw}]_{v,w \in E^0}$. It suffices to show 
    \begin{equation}\label{eqn:lemmaharmonic}
        \sum_{w \in E^0} a_{vw}\harm_w = e^\beta \harm_v
    \end{equation}
    for any $v \in E^0$. Note that the left hand side of Equation \eqref{eqn:lemmaharmonic} is equal to $\sum_{w \in S} a_{vw}\tilde{\harm}_w$ by definition of $\harm$. If $v \in S$, then 
    \begin{equation}
        \sum_{w \in S} a_{vw}\tilde{\harm}_w = e^\beta \tilde{\harm}_v = e^\beta \harm_v
    \end{equation}
    since $\tilde{\harm}$ is a normalised $\beta$-harmonic vector of $E_S$. Hence Equation \eqref{eqn:lemmaharmonic} holds in this case. If $v \not \in S$, then $a_{vw} = 0$ for any $w \in S$ by assumption, so that we have
    \begin{equation}
        \sum_{w \in S} a_{vw}\tilde{\harm}_w = 0 = e^\beta \harm_v. 
    \end{equation}
    Hence, Equation \eqref{eqn:lemmaharmonic} is also true in this case. 
\end{proof}

First, we show the existence of KMS$_\beta$ states for all negative $\beta$ in this setting. 
Recall that a subset $T$ of $E^0$ is called \emph{saturated} if $v \in E^0$ satisfies $r(vE^1) \subseteq T$, then $v \in T$. For a hereditary subset $H$ of $E^0$, the smallest saturated subset of $E^0$ containing $H$ is called the \emph{saturation} of $H$. If $E$ contains no sinks, then the saturation of $H$ coincides with the set of all vertices in $E^0$ such that there exists $n \in \Nz$ with $r(vE^n) \subseteq H$ (see \cite[p.6]{aHLRS15}). The graph $E$ is called \emph{cofinal} if there are no saturated hereditary subsets of $E^0$ other than $E^0$ and $\emptyset$.

\begin{proposition} \label{prop:KMSexistence}
	For any $\beta <0$, there is at least one KMS$_\beta$ state on $\cT C^*(E)$ supported in $E^\infty$ such that the associated $e^\beta$-conformal measure has support containing the orbit $[x_E]$.
\end{proposition}

\begin{proof}
    Fix $\beta<0$. Let $S=V_{x_E}$, and we apply \cite[Theorem 4.2 (a)]{Th17Chapter} to $E_S$. We need to check the following conditions in \cite[p.281 and p.288]{Th17Chapter}:
    \begin{enumerate}[\upshape(i)]
        \item The graph $E_S$ is cofinal. 
        \item The graph $E_S$ contains no sinks. 
        \item For any $v,w \in E_S$ and any $n\geq 1$, $|vE_S^nw|<\infty$. 
        \item For any $v \in E_S$, $vE_S^*v=\emptyset$. 
    \end{enumerate}
    The condition (ii) follows from the definition of $S=V_{x_E}$. The conditions (iii) and (iv) follow from the assumptions on $E$. We show that $E_S$ satisfies (i). Let $T$ be a nonempty saturated hereditary subset of $E_S^0$. Then, we have $H_n :=\{v_n,v_{n+1},\cdots\} \subseteq T$ for large $n$, since any vertex in $E_S^0$ is connected to $H_n$ for large $n$. We can see that the saturation of $H_n$ in $E_S$ is equal to $E^0_S$ by assumptions (2) and (4) on the graph $E$. Hence, $T=E^0$, which implies that $E_S$ is cofinal. Hence, there is a $\beta$-harmonic vector $\harm$ of $E_S$ by \cite[Theorem 4.2 (a)]{Th17Chapter}. 

    We show $\sum_{v \in S} \harm_v < \infty$, which implies that the $e^\beta$-conformal measure associated to $\harm$ is finite. By the assumption (4) on the graph $E$, it suffices to show $\sum_{n=0}^\infty \harm_{v_n} < \infty$. By Definition \ref{def:harmonicvector}, we have 
    \[ e^\beta \chi_{v_n} = \sum_{w \in S} |v_n E^1 w|\harm_w \geq \harm_{v_{n+1}} \]
    for any $n \in \Nz$, so that 
    $e^{n\beta}\harm_{v_0} \geq \harm_{v_n}$
    for any $n \in \Nz$. The claim follows from this inequation. 

    Hence, we can normalise the harmonic vector $\harm$. The normalisation of $\harm$ extends to a normalised harmonic vector of $E$ by Lemma \ref{lem:harmonicvectorextension}, which gives rise to a KMS$_\beta$ state of $\cT C^*(E)$ satisfying the desired properties. 
\end{proof}

We will see that in several concrete examples (Example \ref{exm:negative} and Example \ref{exm:negative2}), the $e^\beta$-conformal measure is unique, although uniqueness is not guaranteed in general.
We shall establish a criterion which determines whether the $e^\beta$-conformal measure is type I or not. 
An intuitive explanation of Theorem \ref{thm:summable} is as follows: if there is a path which leaves from the path $x$, then that path should return to $x$ very quickly. More precisely, if $\mu$ is a path consisting of edges other than the $x_i$'s, and if $s(\mu) = v_{n_0}, r(\mu)=v_{n_1}$, then $|\mu|$ should be much smaller than $n_1-n_0$. The finiteness condition in \eqref{eqn:summable} is a condition on the ``return speed'' of the paths starting from vertices in $x$.

We let $\cS$ denote the set of sources in $E_S$. 
\begin{theorem}
	\label{thm:summable}
	The orbit $[x_E]$ is $\beta$-summable if and only if 
	\begin{equation} 	
		\lim_{n \to \infty} \sum_{\mu \in \cS E^*v_n} e^{\beta(n-|\mu|)} < \infty,\label{eqn:summable}	
	\end{equation}
	where $\cS E^*v_n:=\{\mu\in E^* : s(\mu)\in \cS, r(\mu)=v_n\}$.
\end{theorem}

The ``only if'' direction of Theorem~\ref{thm:summable} is trivial since $\sum_{\mu \in \cS E^*v_n} e^{\beta(n-|\mu|)}<\sum_{z\in [x_E]}l_{x_E}(z)^\beta$ for every $n\geq 0$. For the ``if'' direction, we need some preliminaries.

The next result is inspired by \cite[Lemma 5.1]{Th17Adv} and Thomsen's notion of ``exit measure''. Given $x=x_0x_1...\in E^\infty$ and $k\geq 0$, we let $x[0,...,k]:=x_0x_1...x_k$.

\begin{lemma}
	\label{lem:measure[x]}
	In addition to the assumptions in Theorem~\ref{thm:summable}, assume that $\cS=\{v_0\}$. 
	Then the limit 
	\begin{equation} 
		\eta_v := \lim_{n\to\infty}e^{\beta n}\sum_{\mu \in vE^*v_n} e^{-\beta |\mu|} \label{eqn:m(v)}
	\end{equation} 
	converges for all $v \in E^0$, which gives rise to a $\beta$-harmonic vector $\eta$. 
	Let $m$ be the  $e^\beta$-conformal measure on $E^\infty$ associated to $\eta$. 
	Then $m$ is finite and $m([x_E])>0$.
\end{lemma}
\begin{proof}
	For every $v \in E^0$, the sequence $\displaystyle e^{\beta n}\sum_{\mu \in vE^*v_n} e^{-\beta |\mu|}$ is eventually increasing:  We have
	\[\begin{aligned}		
		e^{\beta (n+1)}\sum_{\mu \in vE^*v_{n+1}}e^{-\beta |\mu|}
		\geq e^{\beta(n+1)}\sum_{\mu \in vE^*v_n}\sum_{\nu \in v_nE^*v_{n+1}} e^{-\beta |\mu\nu|}
		&\geq e^{\beta(n+1)}\sum_{\mu \in vE^*v_n}e^{-\beta (|\mu|+1)}\\
		&= e^{\beta n}\sum_{\mu \in vE^*v_n}e^{-\beta |\mu|}.
	\end{aligned}\]

	Thus the limit in \eqref{eqn:m(v)} converges by assumption \eqref{eqn:summable}. 
	
	Observe that $\eta_v=0$ whenever $v\not\in V_{x_E}$, so it is easy to see that the equation in Definition \ref{def:harmonicvector} holds for all $v\not\in V_{x_E}$. For $v\in V_{x_E}$, we have
	\[
	\begin{aligned}
		e^{-\beta} \sum_{w \in E^0}|vE^1w| \eta_w 
		= \sum_{e\in vE^1}e^{-\beta}\eta_{r(e)}
		&=\lim_{n \to\infty}e^{\beta n}\sum_{e\in vE^1}\sum_{\nu\in r(e)E^* v_n}e^{-\beta(|\nu|+1)} \\
		&=\lim_{n \to\infty}e^{\beta n}\sum_{\mu\in vE^* v_n}e^{-\beta|\mu|}
		=\eta_v,
	\end{aligned}
	\]
\flushleft where the third equality uses that $v \neq v_n$ for sufficiently large $n$. Hence, $\eta$ is a $\beta$-harmonic vector.
	
	For each fixed natural number $n$, we have
	\[ \eta_{v_n} =e^{-\beta} \sum_{w \in E^0} |v_{n} E^1 w| \eta_w \geq 
	e^{-\beta}|v_n E^1 v_{n+1}| \eta_{v_{n+1}} \geq e^{-\beta} \eta_{v_{n+1}} \] 
	and hence we have $\eta_{v_0} \geq e^{-\beta n} \eta_{v_n}$ for all $n$. Therefore, 
	$\sum_{n=0}^\infty \eta_{v_n} < \infty$. 
	So, the associated $e^\beta$-conformal measure $m$ is finite because of the finiteness of $V_{x_E} \setminus W_{x_E}$. 
	
	Finally, since there is a path from $v_n$ to $v_m$ of length $m-n$ if $m>n$, 
	we have 
	\[ \eta_{v_n} = \lim_{m \to \infty} e^{\beta m}\sum_{\mu\in v_nE^*v_m} e^{-\beta |\mu|}  \geq e^{\beta n}\]
	for every $n$.
	Since $\{x_E\}=\bigcap_{n=0}^\infty\cZ'(x_E[0,...,n-1])$, we have 
	\[
	m(\{x_E\})=\lim_{n\to\infty}m(\cZ'(x_E[0,...,n-1]))=\lim_{n\to\infty}e^{-\beta n}m(\cZ'(v_n))\geq 1,
	\] 
	so that $m([x_E])>0$. 
\end{proof}

\begin{proof}[Proof of Theorem~\ref{thm:summable}]
	We first argue that it is enough to consider the case where $\cS=\{v_0\}$. We construct a new graph $E'$ by identifying $v_0$ and all $v \in \cS \setminus \{v_0\}$ such that $vE^*v_n \neq \emptyset$ for some $n$. 
	Let $x_{E'} := f_0f_1\cdots$ in $E'$. Then, $[x_E]$ is $\beta$-summable in $E$ if and only if $[x_{E'}]$ is $\beta$-summable in $E'$, 
	since the summation in the definition of $\beta$-summability is the same in both graphs. Similarly, condition (\ref{eqn:summable}) holds in $E$ if and only if it holds in $E'$. Here, we used that $V_{x_E}\setminus W_{x_E}$ is finite.
	
	Now assume that $\cS=\{v_0\}$. Let $m$ be the finite $e^\beta$-conformal measure from Lemma~\ref{lem:measure[x]}. Then the normalisation of the restriction of $m$ to $[x_E]$ is an $e^\beta$-conformal probability measure concentrated on $[x_E]$, so that $[x_E]$ is $\beta$-summable.
\end{proof}

\begin{example}
	\label{exm:negative}	
	For each sequence $a = (a_i)_{i=0}^\infty$ of non-negative integers, let $E_a$ be the graph depicted by
	
	
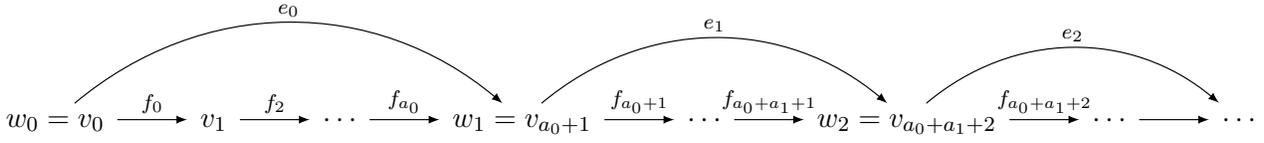
\begin{tikzcd}[
		arrow style=tikz,
		>={latex} 
		]	
		w_0=v_0 \arrow[rrr,"e_0", bend left=40] \arrow[r,"f_0"]  &  v_1\arrow[r,"f_2"]  & \cdots \arrow[r,"f_{a_0}"]& w_1=v_{a_0+1}\arrow[rr,"e_1", bend left=40]\arrow[r,"f_{a_0+1}"]  & \cdots\arrow[r,"f_{a_0+a_1+1}"]&  w_2=v_{a_0+a_1+2}  \arrow[r,"f_{a_0+a_1+2}"] \arrow[rr,"e_2", bend left=40]&\cdots \arrow[r] &\cdots& 
	\end{tikzcd}
	
	which has vertex set $E^0_a=\{v_n : n\geq 0\}$ and edge set $E^1_a=\{e_n,f_n: n\geq 0\}$. Consider the infinite path $x_E=f_0f_1f_2...\in E_a^\infty$.

	Intuitively, $[x_E]$ is $\beta$-summable if and only if the growth speed of $a_i$ is appropriately fast. 
	We claim that for any  $\beta\in\Rz$ and $a = (a_i)_{i=0}^\infty$, the orbit $[x_E]$ is $\beta$-summable in $E_a$ if and only if the product $\prod_{n=0}^{\infty}(1+e^{a_n\beta})^{-1}$ converges to a positive real number.
	
	Let $E=E_a$. Fix a natural number $n$. 
	For a subset $I \subseteq \{0,\cdots,n-1\}$, let $t(I) = \sum_{i\in I} a_i +n$. Let $m=t(\{0,\cdots,n-1\})$. 
	Then, $w_n=v_m$. A path $\mu \in v_0E^*v_m$ is determined by choosing, at each step, either an $e_i$ or the composition of $f_i$'s between $w_n$ and $w_{n+1}$. 
	Hence, there is a bijection between $v_0E^*v_m$ and the power set of $\{0,\cdots,n-1\}$, and the path corresponding to $I$ has length $t(I^c)$. 
	Therefore, 
	\[
	\sum_{\mu \in v_0 E^* v_m} e^{\beta(m-|\mu|)} = \sum_{I} e^{\beta(m-t(I^c))} = \sum_I e^{\beta(\sum_{i \in I}a_i)} = \sum_I \prod_{i \in I} e^{a_i\beta} = \prod_{j=0}^{n-1} (1+e^{a_j\beta}). 
	\]
	Hence, we conclude that $\displaystyle \lim_{m \to \infty} \sum_{\mu \in v_0 E^* v_m} e^{\beta(m-|\mu|)} < \infty$ if and only if 
	$\displaystyle \prod_{n=0}^{\infty} \frac{1}{1+e^{a_n\beta}} > 0$, which completes the proof by Theorem~\ref{thm:summable}.  
	
	Moreover, we can see the uniqueness of the KMS$_\beta$ state of $\T C^*(E)$ supported in $E^\infty$ for each $\beta <0$. Let $\harm$ be a $\beta$-harmonic vector. 
	Then, $\harm_{w_i}$ is uniquely determined from $\harm_{w_0}$ since $\harm_{w_i} = (e^{-\beta} + e^{-a_i \beta}) \harm_{w_{i+1}}$. 
	If $v_i \not\in \{w_j \mid j=0,1,2,\cdots \}$, then $\harm_{v_i}$ is determined from $w_j$'s by the equation $\harm_{v_i} = e^{-\beta}\harm_{v_{i+1}}$. In addition, the value of $\harm_{w_0}$ is determined from the normalisation condition $\sum_{i=0}^\infty \harm_{v_i} = 1$. Therefore, the $\beta$-harmonic vector is unique.
	
	In conclusion, there is a unique KMS$_\beta$ state on $\cT C^*(E)$ for any $\beta <0$ (see Proposition \ref{prop:KMSexistence}), and the unique KMS$_\beta$ state is of type I if and only if the infinite product $\prod_{n=0}^{\infty}(1+e^{a_n\beta})^{-1}$ converges to a positive real number. 
\end{example}

\begin{example}
	\label{exm:negative2}
	Let $E$ be the directed graph defined by $E^0 = \{v_0, v_1, \cdots \}$ and $E^1=\{e_0,e_1,...\}\cup\{f_0,f_1,...\}$, where $s(f_n)=v_n$, $r(f_n)=v_{n+1}$, $s(e_n)=v_n$, and $r(e_n)=v_{n+2}$ for all $n\geq 0$. Note that $E$ can be depicted as
	
	
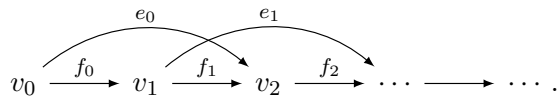
\begin{center}
		\begin{tikzcd}[
			arrow style=tikz,
			>={latex} 
			]
			& v_0 \arrow[rr,"e_0", bend left=40] \arrow[r,"f_0"]  &  v_1\arrow[rr,"e_1", bend left=40]\arrow[r,"f_1"]  & v_2\arrow[r,"f_{2}"]  & \cdots\arrow[r]&\cdots .
		\end{tikzcd}
	\end{center}
	
	Let $x_E=f_0f_1f_2...\in E^\infty$. Intuitively, this graph has ``bounded returning speed", which implies non-$\beta$-summability. In order to show that $[x_E]$ is not $\beta$-summable, let $\tilde{E}$ be the graph obtained by removing $e_n$ for odd $n$ from $E$. Then $\tilde{E} = E_a$, where 
	$a = (a_i)_{i=0}^\infty$ with $a_i = 2$ for all $i$ and $E_a$ is the graph in Example \ref{exm:negative}. Since $x_{\tilde{E}} = x_E$ is not $\beta$-summable in $E_a$, we have
	\[ \sum_{\mu \in \cS E^*v_n} e^{\beta(n-|\mu|)} \geq \sum_{\mu \in \cS \tilde{E}^*v_n} e^{\beta(n-|\mu|)} \longrightarrow \infty, \]
	which implies that $x_E$ is not $\beta$-summable in $E$. 
	
	From Proposition \ref{prop:KMSexistence}, there is at least one normalised $\beta$-harmonic vector for $E$. We will show that this is the unique normalised $\beta$-harmonic vector for $E$. Consequently, the Toeplitz algebra of $E$ has a unique KMS$_\beta$ state supported in $E^\infty$ for each $\beta < 0$. 
	
	Let $\harm$ be a $\beta$-harmonic vector, and let $\harm_n :=\harm_{v_n}$. Then, $\harm$ satisfies the recurrence relation
	\[ \harm_n = e^{-\beta} (\harm_{n+1} + \harm_{n+2}). \]
	Hence, there exist constants $A,B \in \Rz$ such that $\harm_n = A\lambda_1^n + B \lambda_2^n$, where 
	\[ \lambda_1 = \frac{-1 + \sqrt{1+4e^\beta}}{2},\ \lambda_2 = \frac{-1 - \sqrt{1+4e^\beta}}{2}. \]
	Since $|\lambda_2|>|\lambda_1|$ and $\harm_n \geq 0$ for all $n$, we can see that $B=0$. 
	Since $\displaystyle 0< \lambda_1 < \frac{-1+\sqrt{5}}{2} < 1$, 
	we have $\displaystyle \sum_{n=0}^\infty A \lambda_1^n = \frac{A}{1-\lambda_1}$. Hence $A=1-\lambda_1$, which completes our proof of uniqueness.
\end{example}

\appendix
\section{Factoriality of extremal KMS and ground states}

Here, we make some observations about factoriality and purity of KMS and ground states of C*-dynamical systems of possibly non-unital C*-algebras. These observations are likely well-know to experts, but we collect them here for the reader's convenience.

\begin{lemma}
 \label{lem:extremal}
Let $(A,\sigma)$ be a C*-dynamical system and $\beta\in\Rz^*$.
    \begin{enumerate}[\upshape(i)]
        \item A state $\varphi\in\KMS_\beta(A,\sigma)$ is extremal if and only if it is a factor state.
        \item A state $\varphi\in\Gr(A,\sigma)$ is extremal if and only if it is a pure state.
    \end{enumerate}
\end{lemma}
\begin{proof}
Let $\tilde{A}$ denote the unitisaton of $A$, and let $\tilde{\sigma}$ denote the unique extension of $\sigma$ to a time evolution on $\tilde{A}$. Let $S(A)$ and $S(\tilde{A})$ denote the state spaces of $A$ and $\tilde{A}$, respectively. Each $\phi\in  S(A)$ has a unique extension to a state $\tilde{\phi}\in S(\tilde{A})$ such that $\tilde{\phi}(z 1+a)=z+\phi(a)$ for all $z\in\Cz$ and $a\in A$. A short argument shows that if $\phi$ is in $\KMS_\beta(A,\sigma)$ (resp. $\Gr(A,\sigma)$), then $\tilde{\phi}$ is in $\KMS_\beta(\tilde{A},\tilde{\sigma})$ (resp. $\Gr(\tilde{A},\tilde{\sigma})$).

The image $\cF$ of $S(A)$ in $S(\tilde{A})$ is a face (cf. \cite[Chapter~III,~Proposition~6.27]{Tak}), so that for every convex set $C\subseteq S(\tilde{A})$, the intersection $\cF\cap C$ is either empty or it is a face of $C$. Hence, by taking $C=\KMS_\beta(\tilde{A},\tilde{\sigma})$ (resp. $C=\Gr(\tilde{A},\tilde{\sigma})$), we see that $\phi$ is extremal in $\KMS_\beta(A,\sigma)$ (resp. $\Gr(A,\sigma)$) if and only if $\tilde{\phi}$ is extremal in $\KMS_\beta(\tilde{A},\tilde{\sigma})$ (resp. $\Gr(\tilde{A},\tilde{\sigma})$). By taking $C=S(\tilde{A})$, we see that a state $\phi\in S(A)$ is pure if and only if $\tilde{\phi}\in S(\tilde{A})$ is pure, so (ii) follows from \cite[Theorem~5.3.37]{BR}.

The GNS-triple for $\phi$ is the restriction of the GNS-triple for $\tilde{\phi}$, so $\pi_\phi(A)''\cong\pi_{\tilde{\phi}}(\tilde{A})''$. Thus, (i) follows from \cite[Theorem~5.3.30]{BR}.
\end{proof}

\begin{remark}
The analogue of Lemma~\ref{lem:extremal} fails for KMS$_0$ states. For instance, the Toeplitz algebra $\cT$ equipped with the gauge action has a unique KMS$_0$ state, which is given by composing the canonical surjection $\cT\to C(\Tz)$ with the state coming from the (normalised) Haar measure on $\Tz$. However, this state is not a factor state.
\end{remark}

\end{document}